\documentclass[11pt]{article}
\usepackage{latexsym,amssymb,amsmath}
\usepackage[notref,notcite]{showkeys}
\usepackage{amssymb,amsthm,upref,amscd}
\usepackage{color}
\usepackage[T1]{fontenc}
\begin{document}
\baselineskip=15pt

\numberwithin{equation}{section}

\newtheorem{thm}{Theorem}[section]
\newtheorem{lem}[thm]{Lemma}
\newtheorem{cor}[thm]{Corollary}
\newtheorem{Prop}[thm]{Proposition}
\newtheorem{Def}[thm]{Definition}
\newtheorem{Rem}[thm]{Remark}
\newtheorem{Ex}[thm]{Example}
\newtheorem{Ass}[thm]{Assumption}

\newcommand{\A}{\mathbb{A}}
\newcommand{\B}{\mathbb{B}}
\newcommand{\C}{\mathbb{C}}
\newcommand{\D}{\mathbb{D}}
\newcommand{\E}{\mathbb{E}}
\newcommand{\F}{\mathbb{F}}
\newcommand{\G}{\mathbb{G}}
\newcommand{\I}{\mathbb{I}}
\newcommand{\J}{\mathbb{J}}
\newcommand{\K}{\mathbb{K}}
\newcommand{\M}{\mathbb{M}}
\newcommand{\N}{\mathbb{N}}
\newcommand{\Q}{\mathbb{Q}}
\newcommand{\R}{\mathbb{R}}
\newcommand{\T}{\mathbb{T}}
\newcommand{\U}{\mathbb{U}}
\newcommand{\V}{\mathbb{V}}
\newcommand{\W}{\mathbb{W}}
\newcommand{\X}{\mathbb{X}}
\newcommand{\Y}{\mathbb{Y}}
\newcommand{\Z}{\mathbb{Z}}
\newcommand\ca{\mathcal{A}}
\newcommand\cb{\mathcal{B}}
\newcommand\cc{\mathcal{C}}
\newcommand\cd{\mathcal{D}}
\newcommand\ce{\mathcal{E}}
\newcommand\cf{\mathcal{F}}
\newcommand\cg{\mathcal{G}}
\newcommand\ch{\mathcal{H}}
\newcommand\ci{\mathcal{I}}
\newcommand\cj{\mathcal{J}}
\newcommand\ck{\mathcal{K}}
\newcommand\cl{\mathcal{L}}
\newcommand\cm{\mathcal{M}}
\newcommand\cn{\mathcal{N}}
\newcommand\co{\mathcal{O}}
\newcommand\cp{\mathcal{P}}
\newcommand\cq{\mathcal{Q}}
\newcommand\rr{\mathcal{R}}
\newcommand\cs{\mathcal{S}}
\newcommand\ct{\mathcal{T}}
\newcommand\cu{\mathcal{U}}
\newcommand\cv{\mathcal{V}}
\newcommand\cw{\mathcal{W}}
\newcommand\cx{\mathcal{X}}
\newcommand\ocd{\overline{\cd}}

\def\c{\centerline}
\def\ov{\overline}
\def\emp {\emptyset}
\def\pa {\partial}
\def\bl{\setminus}
\def\op{\oplus}
\def\sbt{\subset}
\def\un{\underline}
\def\al {\alpha}
\def\bt {\beta}
\def\de {\delta}
\def\Ga {\Gamma}
\def\ga {\gamma}
\def\lm {\lambda}
\def\Lam {\Lambda}
\def\om {\omega}
\def\Om {\Omega}
\def\sa {\sigma}
\def\vr {\varepsilon}
\def\va {\varphi}

\title{\bf Stability and instability of standing waves for a generalized  Choquard equation with potential}

\author{ Zhipeng Cheng,\ \ Minbo Yang\thanks{Minbo Yang is the corresponding author, this work is partially supported by NSFC (11571317, 11101374, 11271331) and ZJNSF(LY15A010010), mbyang@zjnu.edu.cn}\
\\
\\
{\small Department of Mathematics, Zhejiang Normal University} \\ {\small  Jinhua, Zhejiang, 321004, P. R. China}}

\date{}
\maketitle

\begin{abstract}
We are going to study the standing waves for a generalized  Choquard equation with potential:
$$
-i\partial_t u-\Delta u+V(x)u=(|x|^{-\mu}\ast|u|^p)|u|^{p-2}u, \ \ \hbox{in}\ \  \mathbb{R}\times\mathbb{R}^3,
$$
where $V(x)$ is a real function, $0<\mu<3$,  $2-\mu/3<p<6-\mu$ and $\ast$ stands for convolution.  Under suitable assumptions on the potential and appropriate frequency $\omega$ ,  the stability and instability of the standing waves $u=e^{i \omega t}\varphi(x)$ are investigated .
 \vspace{0.3cm}

\vspace{0.3cm}

 \noindent {\bf Keywords:}  Stability; Instability;  standing  wave;  ground state solution; generalized  Choquard equation.
\end{abstract}

\section{Introduction and main results}

In this paper, we are going to study the following nonlocal Schr\"odinger equation :

\begin{equation}\label{ME}
\left\{\begin{array}{l}
\displaystyle -i\partial_t u-\Delta u+V(x)u=(|x|^{-\mu}\ast|u|^p)|u|^{p-2}u, \ \ \hbox{in}\ \  \mathbb{R}\times\mathbb{R}^3,\\
\displaystyle u(0,x) = u_0(x),
\end{array}
\right.
\end{equation}
where $V(x)$ is a real valued function, $0<\mu<3$ and $2-\mu/3<p<6-\mu$.
This equation arises in physics as an effective description of a non-relativistic bosonic system with two-body interactions in its mean field limit, it is also known to describe the propagation of electromagnetic
waves in plasmas \cite{BC} and plays an important
role in the theory of Bose-Einstein condensation \cite{D}. This equation, which is also called the Hartree equations or the Schr\"odinger-Newton equations, has attracted a great deal of attention in theoretical over the past years.

As we all know, the Cauchy problem of nonlinear generalized  Choquard equation has been intensively studied since the pioneering work by Chadam
and Glassey in \cite{ChG}.  We refer the readers to \cite{C, FTY, GV} for a complete overview of the literature on the topic of the Cauchy problem and asymptotic behavior of the solutions. In this paper we are interested in the standing wave type solutions, i.e.  solutions of the form
\begin{equation}\label{SWS}
u(x)=e^{i \omega t}\varphi(x),
\end{equation}
where $\omega > 0$,$\varphi\in H^1(\mathbb{R}^3)\setminus\{0\}$ satisfies the following nonlocal elliptic equation:
\begin{equation}\label{SE}
-\Delta \varphi+\omega\varphi+V(x)\varphi-(|x|^{-\mu}\ast |\varphi|^p)|\varphi|^{p-2}\varphi=0.
\end{equation}
For the nonlocal Schr\"odinger equations with $V(x)\equiv0$,   Chen and Guo  in \cite{CG} studied
\[
\left\{\begin{array}{l}
\displaystyle i\varphi_t + \Delta\varphi+\big(\frac{1}{|x|^\alpha}\ast|\varphi|^p\big)|\varphi|^{p-2}\varphi = 0,\\
\displaystyle  \varphi(x,0) = \varphi_0(x),\hspace{5.0mm}x\in\mathbb{R}^3.
\end{array}
\right.
\]
and proved the instability of the standing wave solution. For the Cauchy
problem of the Hartree equation with harmonic potential, that is  $V(x)=|x|^2$, we
refer the readers to book \cite{C} and the references therein. In  \cite {CGH},  the authors  derived a variant of Gagliardo-
Nirenberg interpolation inequality involving nonlocal nonlinearity and determined its best (smallest)
constant.  The authors also established a sharp criterion for the global existence
and blow-up of solutions of the  Hartree equation with harmonic potential.
In \cite{Wa} the author obtained the blow-up  and  strong instability result via construction of a cross-constrained invariant set. While in \cite{CWZ}, the authors studied the ground states  of
 $$
-\Delta u+\omega u=(|x|^{-1}\ast|u|^2)u,\ \ \hbox{in}\ \ \mathbb{R}^3.
$$
and considered the stability of the standing waves for a class of Hartree equation. In \cite{CMS}, the classical limit of \eqref{ME} with harmonic potential and $p=2$ was studied by Carles, Mauser and Stimming.  We would like to mention that a recent paper \cite{dS} where  the authors investigated the soliton dynamics for the Hartree equation by proving stability estimates in the spirit of  Weinstein for local equations. subsequently, without using the uniqueness and nondegeneracy of the ground states of the generalized Choquard equation, the authors in \cite{BdGS}  also  studied the soliton dynamics behavior.
For the stability and instability of standing wave of nonlinear local Schr\"odinger equation, we may refer the readers to  \cite{BF, CL,F1,F2,CAS,S,GSS1,GSS2}.

The aim of this paper is to consider the instability and stability of standing waves for a class of  generalized  Choquard equation with potentials, including  harmonic cases. Suppose that $V(x)$ is a radial function and satisfies the the following conditions:
\begin{itemize}
\item[$(V0).$]  There exist two radial functions $V_1(x)$ and $ V_2(x)$ such that $V(x)=V_1(x) + V_2(x)$.
\item[$(V1.1).$]  $V_1(x)\in C^2(\mathbb{R}^3)$ and  there exist positive constants $ m$, $M$ such  that  $ 0 \leq V_1(x)$ and $\leq M(1+|x|^m)$ on $\mathbb{R}^3$.
\item[$(V1.2).$] There  exists $M_\alpha > 0$ such  that $|x^\alpha\partial_{x}^\alpha V_1(x)|\leq M_\alpha(1 + V_1(x))$ on $\mathbb{R}^3$ for $|\alpha|\leq 2$.
\item[$(V1.3).$] $V_1(x)\in C^\infty(\mathbb{R}^3)$,$V_1(x)$ is positive in $\mathbb{R}^3$ and $\partial_{x}^\alpha V_1(x)\in L^\infty(\mathbb{R}^3)$ for $|\alpha|\geq 2$.
\item[$(V2).$]  There exists $ q \geq 1$ such that $q>3/2$  and $ x^\alpha\partial_{x}^\alpha V_2(x) \in L^q(\mathbb{R}^3)+L^\infty(\mathbb{R}^3)$ for $ |\alpha| \leq 2$.
\end{itemize}
Define the Hilbert space $X$ by
\hspace{3.1cm}$$X \triangleq \{v \in H^1(\mathbb{R}^3,\mathbb{C});V_1(x)|v(x)|^2 \in L^1(\mathbb{R}^3)\}$$
with the inner product
\hspace{2.0cm}$$(v,w)_X \triangleq Re \int_{\mathbb{R}^3} (v(x)\overline{w(x)} + \nabla v(x) \cdot \overline{\nabla w(x)}
+ V_1(x)v(x)\overline{w(x)})dx$$
and the  corresponding norm of $X$ denoted by $\|\cdot\|_X$.

\begin{Prop} (Hardy-Littlewood-Sobolev inequality)\label{HSI1}
Let $0<\mu<n$ and suppose that $f\in L^q(\mathbb{R}^n)$,$h\in L^r(\mathbb{R}^n)$ with $\frac{1}{q}+\frac{1}{r}+\frac{\mu}{n}=2$ and $1<q,r<\infty$, then
$$\int_{\mathbb{R}^n\times\mathbb{R}^n}\frac{|f(x)||h(x)|}{|x-y|^\mu}dxdy\leq C(q,r,\mu,n)\|f\|_{L^q}\|h\|_{L^r},\ \ x,y\in\mathbb{R}^n,$$
where $C(q, r, \mu, n)$ is a positive constant depending on $q$, $r$, $\mu$ and $n$.
\end{Prop}
We  define the energy functional $E$ on $X$ by
$$E(v)\triangleq \frac{1}{2}\|\nabla v\|_{2}^2+\frac{1}{2}\int_{\mathbb{R}^3}V(x)|v(x)|^2dx-\frac{1}{2p}
\mathfrak{F}_\mu(v)$$
with $ \displaystyle\mathfrak{F}_\mu(v)=\int_{\mathbb{R}^3\times\mathbb{R}^3}\frac{|v(x)|^p|v(y)|^p}{|x-y|^\mu}dxdy$.
By assumptions $(V0)-(V2)$ and $2-\mu/3 < p <6-\mu$, applying Proposition \ref{HSI1}, we know that $E(v)$ is well defined on $X$.

In the following we will make the following assumption.
\begin{Prop}\label{CAUP}
\ Let $0<\mu<3$. For any $u_0 \in X$,there exist $T=T(\|u_0\|_X)>0$ and a unique solution $u(t)\in C([0,T],X)$ of \eqref{ME} with $u(0)=u_0$ satisfying $$E(u(t))=E(u_0),\hspace{8mm} \|u(t)\|_{2}^2=\|u_0\|_{2}^2,\hspace{8mm}t\in [0,T].$$
In addition, if $u_0\in X$ satisfies $|x|u_0\in L^2(\mathbb{R}^3)$,then the viral identity
$$\frac{d^2}{dt^2}\|xu(t)\|_{2}^2=8P(u(t)),$$
holds for $t\in[0,T]$, where
\begin{equation}\label{P}
P(v)=\|\nabla v\|_{2}^2-\frac{1}{2}\int_{\mathbb{R}^3}x\cdot\nabla V(x)|v(x)|^2dx-\frac{3(p-2)+\mu}{2p}\mathfrak{F}_\mu(v).
\end{equation}
\end{Prop}
In fact, for  $V(x)=|x|^2$ and $0<\mu<3$,  Chen etal. in \cite {CGH} proved that
$$\frac{d}{dt}\|u(t)\|_{2}^2=0,\ \ \ \ \ \ \ \ \frac{d}{dt}E(u(t))=0$$
and
$$\|u(t)\|_{2}^2=\|u(0)\|_{2}^2\ \ \ \ and\ \ \ \ E(u(t))=E(u(0)).$$
Moreover,  if $2\leq p<2+(2-\mu)/3$, the equation \eqref{ME} exists globally in time for any initial value $u_0\in X$; while for  $2\leq p=2+(2-\mu)/3$, the equation \eqref{ME} exists globally in time provided the initial data $\|u_0\|_{L^2}$ sufficiently small.

In order to state our results, we need to  define on $X$,
$$S_\omega(v)=E(v)+\omega Q(v),\ Q(v)=\frac{1}{2}\|v\|_{2}^2$$
and
$$I_\omega(v)=\|\nabla v\|_{2}^2+\omega\|v\|_{2}^2+\int_{\mathbb{R}^3}V(x)|v(x)|^2dx-
 \displaystyle\mathfrak{F}_\mu(v).$$

Consider the minimization problem as follows
\begin{equation}
S=\inf_{v\in\mathcal{N}}S_\omega(v),
\end{equation}
where
\begin{equation}\label{A8}
\mathcal{N}_\omega=\{v\in X;v\neq 0, \ I_\omega(v)=0\}.
\end{equation}
\begin{Def}
A ground state solution of \eqref{SE} is $\varphi\in H^1(\mathbb{R}^3)$ with $v>0$ and solving
\begin{equation}\label{B1}
S_\omega(\varphi)=\inf_{v\in\mathcal{N}_\omega}S_\omega(v).
\end{equation}
\end{Def}
In the following we will use the notion
\begin{equation}\label{A9}
\mathcal{M}_\omega=\{v\in X;v\neq 0, S'_\omega(v)=0, S_\omega(v)=S\}
\end{equation}
to denote the set of ground state solutions.

\begin{Rem}\label{ppp}
We can assume that there exists $\omega_0\in(0,\ \infty)$ such that $\mathcal{M}_\omega$ is not empty and $\mathcal{M}_\omega\subset\{v\in X_G;|x|v(x)\in L^2(\mathbb{R}^3)$ for any $\omega\in(\omega_0,\ \infty)$(The detail we can see Section 2).
\end{Rem}

Now,we study the stability of the minimizers of \eqref{A8} in the following sense.
\begin{Def}\label{Def}
For $\varphi_\omega\in \mathcal{M}_\omega$, we say that a standing wave solution $e^{i \omega t}\varphi_\omega(x)$ of \eqref{ME} is stable in $X$ if for any $\varepsilon>0$, there exists $\delta>0$ such that $\inf_{\theta\in\mathbb{R}}\|u(0)-e^{i\theta}\varphi_\omega\|_X<\delta$ ,$\theta\in\mathbb{R}$,then the solution $u(t)$ of \eqref{ME} with $u(0)=u_0$ satisfies $$\inf_{\theta\in\mathbb{R}}\|u(t)-e^{i\theta}\varphi_\omega\|_X<\varepsilon\ \ \ for\ any\ t\geq 0.$$
Otherwise, $e^{i\omega t}\varphi_\omega$ is said to be unstable in $X$.
\end{Def}

For instability of standing wave solution of \eqref{ME},we have the following results.
\begin{thm}\label{Instability}
Let $0<\mu<3$, $2+(2-\mu)/3<p<6-\mu$  and assume that conditions $(V0)- (V2)$ hold.
Then there exists $\omega^*_0>\omega_0$ such that for any $\omega\in(\omega^*_0, \infty)$ with $\varphi_\omega(x)\in \mathcal{M}_\omega$, then the standing wave solution $u_\lambda(x,t)=e^{i \omega t}\varphi_\omega(x)$ of \eqref{ME} is unstable in $X$.
\end{thm}

The existence and qualitative properties of solutions of the Choquard equation have been widely studied in the last decades.
In \cite{U1}, Lieb proved the existence and uniqueness, up to translations,
of the ground state of
$$
-\Delta u+\omega u=(|x|^{-1}\ast|u|^2)u,\ \ \hbox{in}\ \ \mathbb{R}^3.
$$
Later, in \cite{Ls}, Lions showed the existence of
a sequence of radially symmetric solutions. In \cite{LL,  MS1} the authors studied
\begin{equation}\label{V0}
-\Delta u+\omega u=(|x|^{-1}\ast|u|^p)|u|^{p-2}u,\ \ \hbox{in}\ \ \mathbb{R}^3
\end{equation} showed the regularity, positivity
and radial symmetry of the ground states and
derived decay property at infinity as well. Generally, the uniqueness and nondegeneracy of the ground states  is not known. In a recent paper \cite{U2}, Xiang considered the  uniqueness of the equation \eqref{V0} and proved the following property of the ground state.

\begin{lem}\label{RGS}
There exists $0<\eta<1/3$ such that for any $p$, $2<p<2+\eta$, there exists a unique positive radial ground state $\psi_1\in H^1(\mathbb{R}^3)$ for equation
\begin{equation}\label{UN}
-\Delta u+ u=(|x|^{-1}\ast|u|^p)|u|^{p-2}u,\ \ \hbox{in}\ \ \mathbb{R}^3.
\end{equation}
\end{lem}

By Remark \ref{ppp},  we have the following stability results for the standing waves of  equation \eqref{ME}.
\begin{thm}\label{MR}
Assume that conditions $(V0)-(V2)$ hold and $2<p<2+\eta'$ for some $0<\eta'<\eta$ where $\eta>0$ is the constant  in Lemma \ref{RGS}.  There exists $\omega_{0}^*>\omega_0$ such that, for any $\varphi_\omega(x)\in\mathcal{M}_\omega$,  the standing wave solution $e^{i \omega t}\varphi_\omega(x)$ of \eqref{ME} is stable in $X_G$ in the sense of definition \ref{Def}.
\end{thm}

In the last decades, many people studied the instability and stability of standing wave solution of local Schr\"odinger equation (see e.g. \cite{BF},\cite{ CL},\cite{CFF}-\cite{GGX}, \cite{CAS},\cite{S},\cite{JLT},\cite{H}):
$$i\partial_t u = -\triangle u + V(x)u + |u|^{p-1}u.$$
The idea of the present paper goes back to the paper \cite{F1, F0} by R. Fukuizumi, there the authors assumed that $V(x)$ satisfying conditions $(V0)- (V2)$ and applied the concentration principle(see \cite{PLL1} and \cite{PLL2}) to study the ground state solution $\varphi_\omega$
of the following elliptic equation:
\begin{equation}\label{FEE}
-\Delta\varphi+V(x)\varphi+\omega\varphi-|\varphi|^{p-1}\varphi=0,
\end{equation}
where $1<p<2^*-1$.  Then, $\widetilde{\varphi}_\omega(x)$, defined by the  scaling of  $\varphi_\omega(x)=\omega^{1/(p-1)}\widetilde{\varphi}_\omega(\sqrt{\omega}x)$,
 is a ground state solution of
\begin{equation}\label{REE2}
-\Delta\varphi+\varphi+\omega^{-1}V(\frac{x}{\sqrt{\omega}})\varphi-|\varphi|^{p-1}\varphi=0.
\end{equation}
Then under suitable assumptions on the $V(x)$, $\omega^{-1}V(\frac{x}{\sqrt{\omega}})\widetilde{\varphi}_\omega\to 0$ in some sense as $\omega\to \infty$. Then for  $p>1+4/n$, inspired the behavior of  the orbit of the standing wave of \eqref{REE2} with $V(x)=0$, he obtained that if the sufficient condition $\partial_{\lambda}^2E(\varphi_{\omega}^\lambda)|_{\lambda=1}<0$ holds, where $\varphi_{\omega}^\lambda=\lambda^{n/2}\varphi_\omega(\lambda x)$ then the ground state  solution of \eqref{FEE} blows up at finite time, i.e. the standing wave solution is instable. Then, using the stability result of the standing wave of the limit problem, he proved that the standing wave of \eqref{FEE} is stable. This type of arguments was used by the authors in \cite{CWZ} to study the stability of the standing waves for a class of Hartree equation with potentials including the harmonic one as a particular case. In the present paper we follow the idea explored in \cite{BF} and \cite{F1} to study the generalized Choquard equation with a general class of potential which also includes the harmonic one as a particular case. Moreover, in our situation, the exponent $p$ lies in a interval close to $2$.

This paper is organized as follows. In section $2$, we will prove some basic properties of the ground state solution of \eqref{ME}. In the first subsection of section $3$, we give a sufficient condition for verifying the instability of standing wave solution of \eqref{ME}. Then, in subsection $3.2$, we give the proof of main Theorem about of the instability. In subsection $4.1$, the sufficient conditions for verifying the stability of standing wave solution of \eqref{ME} with $\mu=1$ is investigated. Finally, we prove the main result about of the stability.

\section{Basic results for the ground states }
 In this section, we will give the definition of the ground state solution of \eqref{SE} and prove the existence of the ground state solution of \eqref{SE}.

We study the functional $S_\omega(v)\in C^1(H^1(\mathbb{R}^3), \R)$ with its derivative  given by
$$\langle S_{\omega}^{'}(v),\phi\rangle=\int_{\mathbb{R}^3}(\nabla v\nabla\phi+\omega v\phi)dx+\int_{\mathbb{R}^3}V(x)v(x)\phi(x)dx-\int_{\mathbb{R}^3\times\mathbb{R}^3}\frac{|v(x)|^{p-1}|v(y)|^{p}}{|x-y|^\mu}\phi(x)dxdy,$$
$\phi(x)\in C_{0}^\infty(\mathbb{R}^3)$. Hence, each critical point of $S_\omega(v)$ is a weak solution of \eqref{SE}.
Let $W=\{v\in X;|x|v(x)\in L^2(\mathbb{R}^3)\}$ then it is easy to see that the embedding $W\hookrightarrow L^{q+1}$ is compact, where $1\leq q<5$. By variational arguments, we know
\begin{Prop}
Let $0<\mu<3$ and $2-\mu/3<p<2+(2-\mu)/3$. For $\omega>0$, there exists $\varphi_\omega\in\mathcal{N}_\omega$ such that
$$S_\omega(\varphi_\omega)=\inf_{v\in\mathcal{N}_\omega}S_\omega(v).$$
\end{Prop}
Then there is a Lagrange multiplier $\lambda$ such that
\begin{equation}\label{B2}
S_{\omega}^{'}(\varphi_\omega)-\lambda I_{\omega}^{'}(\varphi_\omega)=0.
\end{equation}
Multiple both side of  the equation \eqref{B2} by $\varphi_\omega$, we obtain
$$\langle S_{\omega}^{'}(\varphi_\omega),\varphi_\omega\rangle=\lambda\langle I_{\omega}^{'}(\varphi_\omega),\ \varphi_\omega\rangle.$$
Noticing that $$\langle S_{\omega}^{'}(\varphi_\omega),\varphi_\omega\rangle=I_\omega(\varphi_\omega)=0$$ and
$$\langle I_{\omega}^{'}(\varphi_\omega),\varphi_\omega\rangle=-2(p-1) \displaystyle\mathfrak{F}_\mu(\varphi_\omega)<0.$$
We know $\lambda=0$ and  $\varphi_\omega$ is a ground state solution of \eqref{SE}.

\begin{lem}\label{Inf}
Let $0<\mu<3$,  $2-\mu/3<p<6-\mu$. For any $\omega>0$ with $\varphi_\omega\in\mathcal{M}_\omega$, we have
$$
\aligned
\mathfrak{F}_\mu(\varphi_\omega)&=\inf\Big\{\mathfrak{F}_\mu(v);v\in X\setminus\{0\},I_{\omega}(v)=0\Big \}\\
&=\inf\Big\{\mathfrak{F}_\mu(v); v\in X\setminus\{0\},I_{\omega}(v)\leq 0 \Big\},
\endaligned
$$
and
$$
S_{\omega}(\varphi_\omega)=\inf\Big\{S_{\omega}(v);v\in X\setminus\{0\},\mathfrak{F}_\mu(v)=\mathfrak{F}_\mu(\varphi_\omega)\Big\}.
$$
\end{lem}
\begin{proof}
Since
$$
S_{\omega}(v)=\frac12I_{\omega}(v)+\frac{p-1}{2p}\mathfrak{F}_\mu(v), \ \ v\in X,
$$
we know that
$$
\aligned
\frac{p-1}{2p}\mathfrak{F}_\mu(\varphi_\omega)=S_{\omega}(\varphi_\omega)&=\inf\Big\{S_{\omega}(v);v\in X\setminus\{0\},I_{\omega}(v)=0\Big \}\\
&=\inf\Big\{\frac{p-1}{2p}\mathfrak{F}_\mu(v);v\in X\setminus\{0\},I_{\omega}(v)=0\Big \},
\endaligned
$$
i.e.
$$
\mathfrak{F}_\mu(\varphi_\omega)=\inf\Big\{\mathfrak{F}_\mu(v);v\in X\setminus\{0\},I_{\omega}(v)=0\Big \}.
$$
Let $\Gamma_\omega:=\inf\Big\{\mathfrak{F}_\mu(v); v\in X\setminus\{0\},I_{\omega}(v)\leq 0 \Big\}$, it is obvious that
$$
\Gamma_\omega\leq\mathfrak{F}_\mu(\varphi_\omega).
$$
For any $v\in X\setminus\{0\}$ such that $I_{\omega}(v)<0$, there exits $\lambda_0\in (0,1)$ satisfying  $I_{\omega}(\lambda_0v)=0$. Consequently,
we know that
$$
\mathfrak{F}_\mu(\varphi_\omega)\leq \mathfrak{F}_\mu(\lambda_0v)< \mathfrak{F}_\mu(v),
$$
therefore
$$
\mathfrak{F}_\mu(\varphi_\omega)=\Gamma_\omega.
$$

For any $v$ satisfying  $\mathfrak{F}_\mu(v)=\mathfrak{F}_\mu(\varphi_\omega)$, it is easy to see that
$$
I_{\omega}(v)\geq0,
$$
which implies that
$$
S_{\omega}(v)\geq \frac {p-1}{2p}\mathfrak{F}_\mu(v)=S_{\omega}(\varphi_\omega)=S_{\omega}(\varphi_\omega).
$$
Noticing that $\inf\Big\{S_{\omega}(v);v\in X\setminus\{0\},\mathfrak{F}_\mu(v)=\mathfrak{F}_\mu(\varphi_\omega)\Big\}\leq S_{\omega}(\varphi_\omega) $,  we get the conclusion immediately.
\end{proof}

\begin{lem}\label{ES}
Let $G(x)\in L^q(\mathbb{R}^3)+L^\infty(\mathbb{R}^3)$ for some $q$ such that $q>3/2$ and $q\geq 1$.Then,there exists a constant $C>0$ such that
$$|\int_{\mathbb{R}^3}G(x)|v(x)|^2dx|\leq C\|G\|_{L^q+L^\infty}\|v\|_{H^1}^2,\ \ \ \ v\in H^1(\mathbb{R}^3).$$
\end{lem}
\begin{proof}
The proof is very simple,we let $G(x)=g_1(x)+g_2(x)$,$g_1(x)\in L^q(\mathbb{R}^3)$,$g_2(x)\in L^\infty(\mathbb{R}^3)$ and use the H\"older inequality and Sobolev inequality to prove it.
\end{proof}

\section{Instability of standing waves}
\ \ \ \ In the following, for any $\omega>0$ with $\varphi_\omega\in\mathcal{M}_\omega$, we introduce
the re-scaled function
\begin{equation}\label{Res}
\varphi_\omega(x)=\omega^{\frac{5-\mu}{4(p-1)}}\widetilde{\varphi}_\omega(\sqrt{\omega}x).
\end{equation} Then, $\widetilde{\varphi}_\omega(x)$
 is a ground state solution of
\begin{equation}\label{RRE}
-\Delta\varphi+\varphi+\omega^{-1}V(\frac{x}{\sqrt{\omega}})\varphi-
(|x|^{-\mu}\ast |\varphi|^p)|\varphi|^{p-2}\varphi=0.
\end{equation}
Denote by
$$I_{\omega}^{*}(v)=\|\nabla v\|_{2}^2+\|v\|_{2}^2+
\omega^{-1}\int_{\mathbb{R}^3}V(\frac{x}{\sqrt{\omega}})|v(x)|^2dx-
\mathfrak{F}_\mu(v),$$
$$\hspace{-4.8cm}I_{0}(v)=\|\nabla v\|_{2}^2+\|v\|_{2}^2-\mathfrak{F}_\mu(v),$$
and let $\psi_1(x)$  be the ground state solution of
\begin{equation}\label{Con1}
-\Delta \psi+\psi-(|x|^{-\mu}\ast |\psi|^p)|\psi|^{p-2}\psi=0,
\end{equation}
from \cite{MS1}, we know  the regularity,  positivity
and radial symmetry of $\psi_1(x)$ and  it decays asymptotically at infinity.

\subsection{Sufficient conditions for instability}

\begin{lem}\label{Pro}
Let $0<\mu<3$, $2-\mu/3<p<6-\mu$, $\varphi_\omega\in\mathcal{M}_\omega$ for large $\omega$ and assume that conditions $(V0)- (V2)$ hold. Let $\widetilde{\varphi}_\omega(x)$ be the re-scaled function defined by \eqref{Res} and $\psi_1(x)$ be the ground state solution  of  \eqref{Con1}. Then, we have
\begin{itemize}
\item[$(1).$] $ \lim_{\omega\rightarrow\infty}\mathfrak{F}_\mu(\widetilde{\varphi}_\omega )=\mathfrak{F}_\mu(\psi_1);$
\item[$(2).$] $\displaystyle\lim_{\omega\rightarrow\infty}\omega^{-1}\int_{\mathbb{R}^3}V(\frac{x}{\sqrt{\omega}})|\widetilde{\varphi}_\omega(x)|^2dx=0;$
\item[$(3).$] $\lim_{\omega\rightarrow\infty}\|\widetilde{\varphi}_\omega\|_{H^1}^2=\|\psi_1\|_{H^1}^2.$
\end{itemize}
\end{lem}
\begin{proof}
$(1).$ From Lemma \ref{Inf},  we know
$\psi_1(x)$ is a minimizer of
\begin{equation}\label{C2}
\inf\Big\{\mathfrak{F}_\mu(v);v\in X\setminus\{0\},I_0(v)\leq 0\Big\}.
\end{equation}
Similar to the arguments of  Lemma \ref{Inf}, we may assume that  $\widetilde{\varphi}_\omega(x)$ is a minimizer of
\begin{equation}\label{C1}
\inf\Big\{\mathfrak{F}_\mu(v);v\in X\setminus\{0\},I_{\omega}^*(v)\leq 0 \Big\}.
\end{equation}
Notice that  $I_0(\psi_1)=0$, i.e.
$$\|\nabla\psi_1\|_{2}^2+\|\psi_1\|_{2}^2=\mathfrak{F}_\mu(\psi_1).$$
Then,  for any $\theta>1$, we have
\begin{equation}\label{C22}\theta^{-2}I_{\omega}^*(\theta\psi_1)=-(\theta^{2p-2}-1)
\mathfrak{F}_\mu(\psi_1)
+\omega^{-1}\int_{\mathbb{R}^3}V(\frac{x}{\sqrt{\omega}})|\psi_1(x)|^2dx.
\end{equation}
Using $(V1.1)$ and Lemma \ref{ES}, we know
$$
\aligned
|\omega^{-1}&\int_{\mathbb{R}^3}V(\frac{x}{\sqrt{\omega}})|\psi_1(x)|^2dx|\\&\leq
\omega^{-1}\int_{\mathbb{R}^3}V_1(\frac{x}{\sqrt{\omega}})|\psi_1(x)|^2dx+
|\omega^{-1}\int_{\mathbb{R}^3}V_2(\frac{x}{\sqrt{\omega}})|\psi_1(x)|^2dx|\\
&\leq C\omega^{-1}\int_{\mathbb{R}^3}(1+\omega^{-\frac{m}{2}}|x|^m)|\psi_1(x)|^2dx+
(\omega^{-1}+\omega^{\frac{3}{2q}-1})C\|V_2\|_{L^q+L^\infty}\|\psi_1\|_{H^1}^2.
\endaligned
$$
Since $|x|^m|\psi_1(x)|^2\in L^1(\mathbb{R}^3)$ and $q>3/2$,  using the fact that $\psi_1(x)$ decays exponentially at infinity,  we get
\begin{equation}\label{C3}
\lim_{\omega\rightarrow\infty}\omega^{-1}\int_{\mathbb{R}^3}V(\frac{x}{\sqrt{\omega}})|\psi_1(x)|^2dx=0.
\end{equation}
By \eqref{C3} and \eqref{C22}, it is easy to see
 $$\lim_{\omega\rightarrow\infty}\theta^{-2}I_{\omega}^*(\theta\psi_1)=
\lim_{\omega\rightarrow\infty}-(\theta^{2p-2}-1)\mathfrak{F}_\mu(\psi_1)<0.$$
Namely,  for any $\theta>1$, if  $\omega$ is large enough, we have
$$I_{\omega}^*(\theta\psi_1)<0.
$$

Next, since $I_{\omega}^*(\widetilde{\varphi}_\omega)=0$, i.e. $$\|\nabla\widetilde{\varphi}_\omega\|_{2}^2+\|\widetilde{\varphi}_\omega\|_{2}^2=\mathfrak{F}_\mu(\widetilde{\varphi}_\omega)-
\omega^{-1}\int_{\mathbb{R}^3}V(\frac{x}{\sqrt{\omega}})|\widetilde{\varphi}_\omega(x)|^2dx.$$
Then, for any $\theta>1$, we have
\begin{equation}\label{C4}
\aligned
\theta^{-2}I_0(\theta\widetilde{\varphi}_\omega)&=
-(\theta^{2p-2}-1)\mathfrak{F}_\mu(\widetilde{\varphi}_\omega)
-\omega^{-1}\int_{\mathbb{R}^3}V(\frac{x}{\sqrt{\omega}})|\widetilde{\varphi}_\omega(x)|^2dx\\
&\leq-(\theta^{2p-2}-1)\mathfrak{F}_\mu(\widetilde{\varphi}_\omega)
+\omega^{-1}\int_{\mathbb{R}^3}V_{-}(\frac{x}{\sqrt{\omega}})|\widetilde{\varphi}_\omega(x)|^2dx,
\endaligned
\end{equation}
where $V_{-}(x)=max\{-V(x),\ 0\}$. From the conditions $(V0)-(V2)$, we have $V_{-}\in L^q+L^\infty$ with $q>3/2$ and $q\geq1$. According to Lemma \ref{ES},  there exists $C>0$ and $q>3/2$ such that
\begin{equation}\label{C5}
\omega^{-1}\int_{\mathbb{R}^3}V(\frac{x}{\sqrt{\omega}})|\widetilde{\varphi}_\omega(x)|^2dx\leq
C(\omega^{\frac{3}{2q}-1}+\omega^{-1})\|V_{-}\|_{L^q+L^\infty}\|\widetilde{\varphi}_\omega\|_{H^1}^2.
\end{equation}
Using $I_{\omega}^*(\widetilde{\varphi}_\omega)=0$ again, we have
$$
\aligned
\|\widetilde{\varphi}_\omega\|_{H^1}^2&\leq
\mathfrak{F}_\mu(\widetilde{\varphi}_\omega)+
\omega^{-1}\int_{\mathbb{R}^3}V_{-}(\frac{x}{\sqrt{\omega}})|\widetilde{\varphi}_\omega(x)|^2dx\\
&\leq\mathfrak{F}_\mu(\widetilde{\varphi}_\omega)+
C(\omega^{\frac{3}{2q}-1}+\omega^{-1})\|V_{-}\|_{L^q+L^\infty}\|\widetilde{\varphi}_\omega\|_{H^1}^2,
\endaligned
$$
which implies
\begin{equation}\label{C6}
(1-C(\omega^{\frac{3}{2q}-1}+\omega^{-1})\|V_{-}\|_{L^q+L^\infty})\|\widetilde{\varphi}_\omega\|_{H^1}^2\leq
\mathfrak{F}_\mu(\widetilde{\varphi}_\omega).
\end{equation}
According to \eqref{C5} and \eqref{C6}, we have
\begin{equation}\label{C7}
\omega^{-1}\int_{\mathbb{R}^3}V_{-}(\frac{x}{\sqrt{\omega}})|\widetilde{\varphi}_\omega(x)|^2dx\leq
\frac{C(\omega^{\frac{3}{2q}-1}+\omega^{-1})\|V_{-}\|_{L^q+L^\infty}}{1-C(\omega^{\frac{3}{2q}-1}+\omega^{-1})\|V_{-}\|_{L^q+L^\infty}}
\mathfrak{F}_\mu(\widetilde{\varphi}_\omega).
\end{equation}
Thus, from \eqref{C4} and \eqref{C7}, we have
$$\theta^{-2}I_0(\theta\widetilde{\varphi}_\omega)\leq -(\theta^{2p-2}-1-
\frac{C(\omega^{\frac{3}{2q}-1}+\omega^{-1})\|V_{-}\|_{L^q+L^\infty}}{1-C(\omega^{\frac{3}{2q}-1}+\omega^{-1})\|V_{-}\|_{L^q+L^\infty}})
\mathfrak{F}_\mu(\widetilde{\varphi}_\omega).$$
Namely, for any $\theta>1$, if  $\omega$ is large enough, we have
$$I_0(\theta\widetilde{\varphi}_\omega)<0.$$

As stated above, we have
\begin{center}
$I_{\omega}^*(\theta\psi_1)<0$ and $I_0(\theta\widetilde{\varphi}_\omega)<0$.
\end{center}
By $I_{\omega}^*(\theta\psi_1)<0$ and \eqref{C1}, we have
\begin{equation}\label{C8}
\mathfrak{F}_\mu(\widetilde{\varphi}_\omega)
\leq\theta^{2p}\mathfrak{F}_\mu(\psi_1),
\end{equation}
while, by $I_0(\theta\widetilde{\varphi}_\omega)<0$ and \eqref{C2}, we have
\begin{equation}\label{C9}
\frac{1}{\theta^{2p}}\mathfrak{F}_\mu(\psi_1)
\leq\mathfrak{F}_\mu(\widetilde{\varphi}_\omega).
\end{equation}
Since  $\theta>1$ is arbitrary, from \eqref{C8} and \eqref{C9}, we have
$$\lim_{\omega\rightarrow\infty}\mathfrak{F}_\mu(\widetilde{\varphi}_\omega)
=\mathfrak{F}_\mu(\psi_1).$$

$(2).$ From $I_{\omega}^*(\widetilde{\varphi}_\omega)=0$, we have
\begin{equation}\label{C99}-(\|\nabla\widetilde{\varphi}_\omega\|_{2}^2+\|\widetilde{\varphi}_\omega\|_{2}^2-
\mathfrak{F}_\mu(\widetilde{\varphi}_\omega))=
\omega^{-1}\int_{\mathbb{R}^3}V(\frac{x}{\sqrt{\omega}})|\widetilde{\varphi}_\omega(x)|^2dx.
\end{equation}
Moreover, by $I_0(\theta\widetilde{\varphi}_\omega)<0$ with $\theta=1$ and $(1)$, we have
\begin{equation}\label{C10}
\limsup_{\omega\rightarrow\infty}I_0(\widetilde{\varphi}_\omega)\leq0.
\end{equation}
For $\omega\rightarrow\infty$, there exists $\theta(\omega)>0$ such that $I_0(\theta(\omega)\widetilde{\varphi}_\omega)=0,$
thus,we have
$$\mathfrak{F}_\mu(\psi_1)\leq
\mathfrak{F}_\mu(\theta(\omega)\widetilde{\varphi}_\omega) =\theta(\omega)^{2p}
\mathfrak{F}_\mu(\widetilde{\varphi}_\omega),$$
which together with conclusion $(1)$ implies that
$$\liminf_{\omega\rightarrow\infty}\theta(\omega)^{2p}\geq\liminf_{\omega\rightarrow\infty}
\frac{\mathfrak{F}_\mu(\psi_1)}{\mathfrak{F}_\mu(\widetilde{\varphi}_\omega)}
=1.$$
Using  $I_0(\theta(\omega)\widetilde{\varphi}_\omega)=0$ and conclusion $(1)$ again, we can obtain
\begin{equation}\label{C11}
\liminf_{\omega\rightarrow\infty}I_0(\widetilde{\varphi}_\omega)=\liminf_{\omega\rightarrow\infty}(\theta(\omega)^{2p-2}-1)
\mathfrak{F}_\mu(\widetilde{\varphi}_\omega)
\geq0.
\end{equation}
From \eqref{C10} and \eqref{C11}, we get
$$\lim_{\omega\rightarrow\infty}I_0(\widetilde{\varphi}_\omega)=0,$$
this  implies that
\begin{equation}\label{C12}
\lim_{\omega\rightarrow\infty}\{\|\nabla\widetilde{\varphi}_\omega\|_{2}^2+\|\widetilde{\varphi}_\omega\|_{2}^2-
\mathfrak{F}_\mu(\widetilde{\varphi}_\omega)\}=0.
\end{equation}
Hence, by \eqref{C99}, we get
$$\displaystyle\lim_{\omega\rightarrow\infty}\omega^{-1}\int_{\mathbb{R}^3}V(\frac{x}{\sqrt{\omega}})|\widetilde{\varphi}_\omega(x)|^2dx=0.
$$

$(3).$ From conclusion $(1)$ and $I_0(\psi_1)=0$, we have
$$\lim_{\omega\rightarrow\infty}\mathfrak{F}_\mu(\widetilde{\varphi}_\omega)
=\mathfrak{F}_\mu(\psi_1)=\|\psi_1\|_{H^1}^2.$$
By \eqref{C12} in the proof of $(2)$, we have
$$\lim_{\omega\rightarrow\infty}\|\widetilde{\varphi}_\omega\|_{H^1}^2=
\lim_{\omega\rightarrow\infty}\{\|\nabla\widetilde{\varphi}_\omega\|_{2}^2+\|\widetilde{\varphi}_\omega\|_{2}^2\}=
\lim_{\omega\rightarrow\infty}\mathfrak{F}_\mu(\widetilde{\varphi}_\omega).$$
Hence,
$$\lim_{\omega\rightarrow\infty}\|\widetilde{\varphi}_\omega\|_{H^1}^2=\|\psi_1\|_{H^1}^2.$$
\end{proof}
\begin{Prop}\label{PP}
Let $0<\mu<3$, $2+(2-\mu)/3<p<6-\mu$  and assume that conditions $(V0)- (V2)$ hold.  Then there exists $\omega^*_0>\omega_0$ such that for any $\omega\in(\omega^*_0, \infty)$, $\varphi_\omega\in\mathcal{M}_\omega$,
 $$\partial_{\lambda}^2E(\varphi_{\omega}^\lambda)|_{\lambda=1}<0,$$  where,
$\varphi_{\omega}^\lambda(x)=\lambda^{3/2}\varphi_\omega(\lambda x)$.
\end{Prop}
\begin{proof}
Let $\varphi_{\omega}^\lambda(x)=\lambda^{3/2}\varphi_\omega(\lambda x)$, by  simple calculation, we have
$$E(\varphi_\omega^\lambda)=\frac{\lambda^2}{2}\|\nabla \varphi_\omega\|_{2}^2+\frac{1}{2}\int_{\mathbb{R}^3}V(\frac{x}{\lambda})|\varphi_\omega(x)|^2dx-
\frac{\lambda^{3(p-2)+\mu}}{2p}\mathfrak{F}_\mu({\varphi}_\omega)$$
and
$$\partial_{\lambda}^2E(\varphi_\omega^\lambda)|_{\lambda=1}=\|\nabla \varphi_\omega\|_{2}^2+\frac{1}{2}\int_{\mathbb{R}^3}\{2x\cdot\nabla V(x)+\sum_{i,j=1}^{n}x_i x_j\partial_i \partial_j V(x)\}|\varphi_\omega(x)|^2dx$$
$$\hspace{3.5cm}-\frac{\{3(p-2)+\mu\}\cdot \{3(p-2)+(\mu-1)\}}{2p}\mathfrak{F}_\mu({\varphi}_\omega).$$
Notice that $\varphi_\omega$ is ground state solution of \eqref{SE}, we have
$$P(\varphi_\omega)=\partial_\lambda S_\omega(\varphi_{\omega}^\lambda)|_{\lambda=1}=\langle S_{\omega}^{'}(\varphi_{\omega}^\lambda),\partial_\lambda \varphi_{\omega}^\lambda|_{\lambda=1}\rangle=0.$$
Thus,from the definition of \eqref{P}, we have
$$\|\nabla\varphi_\omega\|_{2}^2=\frac{1}{2}\int_{\mathbb{R}^3}x\cdot\nabla V(x)|v(x)|^2dx+
\frac{3(p-2)+\mu}{2p}\mathfrak{F}_\mu({\varphi}_\omega).$$
Therefore,
$$\partial_{\lambda}^2E(\varphi_\omega^\lambda)|_{\lambda=1}=\frac{1}{2}\int_{\mathbb{R}^3}\{3x\cdot\nabla V(x)+\sum_{i,j=1}^{3}x_i x_j\partial_i \partial_j V(x)\}|\varphi_\omega(x)|^2dx$$
\begin{equation}\label{C13}
\hspace{2.5cm}-\frac{\{3(p-2)+\mu\}\cdot \{3(p-2)+(\mu-2)\}}{2p}\mathfrak{F}_\mu({\varphi}_\omega).
\end{equation}
In the following, we set
$$V^*(x)=3x\cdot\nabla V(x)+\sum_{i,j=1}^3 x_i x_j \partial_i\partial_j V(x)$$
and
$$V_{k}^*=3x\cdot\nabla V_k(x)+\sum_{i,j=1}^{3}x_i x_j\partial_i\partial_j V_k(x),\ \ \ \ k=1, 2, $$
with
\begin{equation}\label{C14}
V^*(x)=V_{1}^*(x)+V_{2}^*(x).
\end{equation}
By Lemma \ref{ES}, $(2)$, $(3)$ of Lemma\ref{Pro}  and condition $(V2)$, we have
\begin{equation}\label{C15}
\omega^{-1}\int_{\mathbb{R}^3}|V_2(\frac{x}{\sqrt{\omega}})|\widetilde{\varphi}_\omega(x)|^2dx\leq
C(\omega^{\frac{3}{2q}-1}+\omega^{-1})\|V_2\|_{L^q+L^\infty}\|\widetilde{\varphi}_\omega\|_{H^1}^2
\end{equation}
and
\begin{equation}\label{C16}
\lim_{\omega\rightarrow\infty}\omega^{-1}\int_{\mathbb{R}^3}V_1(\frac{x}{\sqrt{\omega}})|\widetilde{\varphi}_\omega(x)|^2dx=0.
\end{equation}
Moreover, from the condition $(V1.2)$, we have
$$\omega^{-1}\int_{\mathbb{R}^3}|V_{1}^*(\frac{x}{\sqrt{\omega}})||\widetilde{\varphi}_\omega(x)|^2dx
\leq C\omega^{-1}\int_{\mathbb{R}^3}(1+V_1(\frac{x}{\sqrt{\omega}}))|\widetilde{\varphi}_\omega(x)|^2dx.$$
Thus, from \eqref{C16} and Lemma $3.2$ (3), we have
\begin{equation}\label{C17}
\lim_{\omega\rightarrow\infty}\omega^{-1}\int_{\mathbb{R}^3}|V_{1}^*(\frac{x}{\sqrt{\omega}})||\widetilde{\varphi}_\omega(x)|^2dx=0.
\end{equation}
And the same as \eqref{C15}, we still have
$$\omega^{-1}\int_{\mathbb{R}^3}|V_{2}^*(\frac{x}{\sqrt{\omega}})|\widetilde{\varphi}_\omega(x)|^2dx\leq
C(\omega^{\frac{3}{2q}-1}+\omega^{-1})\|V_{2}^*\|_{L^q+L^\infty}\|\widetilde{\varphi}_\omega\|_{H^1}^2.$$
and
\begin{equation}\label{C18}
\lim_{\omega\rightarrow\infty}\omega^{-1}\int_{\mathbb{R}^3}|V_{2}^*(\frac{x}{\sqrt{\omega}})||\widetilde{\varphi}_\omega(x)|^2dx=0.
\end{equation}
According to \eqref{C14}, \eqref{C17} and \eqref{C18}, we have
\begin{equation}\label{C19}
\lim_{\omega\rightarrow\infty}\omega^{-1}\int_{\mathbb{R}^3}|V^*(\frac{x}{\sqrt{\omega}})||\widetilde{\varphi}_\omega(x)|^2dx=0.
\end{equation}
From  $(1)$ of Lemma \ref{Pro}, \eqref{C19} and the definition of  $\widetilde{\varphi}_\omega(x)$ , we have
\begin{equation}\label{C20}
\displaystyle\lim_{\omega\rightarrow\infty}\frac{\displaystyle\omega^{-1}\int_{\mathbb{R}^3}V^*(\frac{x}{\sqrt{\omega}})
|\widetilde{\varphi}_\omega|^2dx}
{\mathfrak{F}_\mu(\widetilde{\varphi}_\omega)
}=\lim_{\omega\rightarrow\infty}\frac{\displaystyle\int_{\mathbb{R}^3}V^*(x)|\varphi_\omega|^2dx}
{\mathfrak{F}_\mu({\varphi}_\omega)}=0.
\end{equation}
Since $\varphi_\omega(x)=\omega^{\frac{5-\mu}{4(p-1)}}\widetilde{\varphi}_\omega(\sqrt{\omega}x)$,we have $\widetilde{\varphi}_\omega(x)=\omega^{\frac{\mu-5}{4(p-1)}}\varphi_\omega(\frac{x}{\sqrt{\omega}})$.
$$\frac{\displaystyle\omega^{-1}\int_{\mathbb{R}^3}V^*(\frac{x}{\sqrt{\omega}})
|\widetilde{\varphi}_\omega|^2dx}
{\mathfrak{F}_\mu(\widetilde{\varphi}_\omega)
}=\displaystyle\frac{\displaystyle\omega^{\frac{\mu-5}{2(p-1)}}\omega^{-1}\omega^{\frac{3}{2}}
\int_{\mathbb{R}^3}V^*(x)|\varphi_\omega|^2dx}{\omega^{\frac{(\mu-5)p}{2(p-1)}}\omega^3
\omega^{-\frac{\mu}{2}}\mathfrak{F}_\mu(\varphi_\omega)}=\frac{\displaystyle\int_{\mathbb{R}^3}V^*(x)|\varphi_\omega|^2dx}
{\mathfrak{F}_\mu({\varphi}_\omega)}$$
Since $p>2+(2-\mu)/3$, we have
\begin{equation}\label{C21}
\frac{\{3(p-2)+\mu\}\cdot \{3(p-2)+(\mu-2)\}}{2p}>0.
\end{equation}
By \eqref{C20} and \eqref{C21}, we have
\begin{equation}\label{C221}
\frac{\displaystyle\int_{\mathbb{R}^3}V^*(x)|\varphi_\omega|^2dx}
{\mathfrak{F}_\mu({\varphi}_\omega)}
<\frac{\{3(p-2)+\mu\}\cdot \{3(p-2)+(\mu-2)\}}{2p},
\end{equation}
if $\omega$ is large enough.  From  \eqref{C13} and \eqref{C221} we know  there exists $\omega^*_0>\omega_0$ such that for any $\omega\in(\omega^*_0, \infty)$,
$$\partial_{\lambda}^2E(\varphi_{\omega}^\lambda)|_{\lambda=1}<0.$$
\end{proof}

\subsection{Proof of the main result}
In this section, we are going to give the proof of the main result.

For any $\varphi_\omega\in X$ and $\varepsilon>0$, we define
$$U_\varepsilon(\varphi_\omega)\triangleq\{v\in X;\inf_{\theta\in\mathbb{R}}\|v-e^{i\theta}\varphi_\omega\|_{X}<\varepsilon\}.$$
\begin{lem}\label{31}
Let $\varphi_\omega$ be a ground state solution of \eqref{SE}. If $\partial_{\lambda}^2E(\varphi_{\omega}^\lambda)|_{\lambda=1}<0$, then there exists $\varepsilon>0$, $\delta>0$ and mapping $\lambda:U_\varepsilon(\varphi_\omega)\rightarrow(1-\delta, 1+\delta)$ such that
$$I(v^{\lambda(v)})=0\ \ for\ all\ \ v\in U_\varepsilon(\varphi_\omega).$$
\end{lem}
\begin{proof}
Let $$F(v,\lambda)=I(v^\lambda).$$
Since $\varphi_\omega$ is a minimizer of $S_\omega(v)$ constrained on the manifold $\mathcal{N}_\omega$, then
\begin{equation}\label{D1}
\langle S_{\omega}^{''}(\varphi_\omega)\phi,\phi\rangle\geq0,\ \ for\ \ \langle\varphi_\omega,\phi\rangle=0.
\end{equation}
Next, take $\eta=\partial_\lambda\varphi_{\omega}^\lambda|_{\lambda=1}$, since
$$Q(\varphi_\omega)=Q(\varphi_{\omega}^\lambda)\ \ and\ \ \langle S_{\omega}^{'}(\varphi_\omega),\xi\rangle=0,  \forall \xi\in X,$$
then
\begin{equation}\label{D2}
\langle S_{\omega}^{''}(\varphi_\omega)\eta,\eta\rangle=\partial_{\lambda}^2E(\varphi_{\omega}^\lambda)|_{\lambda=1}<0.
\end{equation}
By \eqref{D1} and \eqref{D2}, we know $\langle\eta,\varphi_\omega\rangle\neq0$ and so
$$\partial_\lambda F(\varphi_\omega,1)=\partial_\lambda I(\varphi_{\omega}^\lambda))|_{\lambda=1}=\langle I^{'}(\varphi_{\omega}),\eta\rangle\neq0.$$
Notice that, $$F_\lambda(\varphi_\omega,1)=I(\varphi_\omega)=0,$$
the implicit function theorem implies the existence of $\varepsilon>0$, $\delta>0$ and a mapping $\lambda:B_\varepsilon(\varphi_\omega)\rightarrow(1-\delta,1+\delta)$ such that
$$I(v^{\lambda(v)})=0\ \ for\ all\ \ v\in B_\varepsilon(\varphi_\omega),$$
the conclusion then follows directly.
\end{proof}
\begin{lem}\label{32}
If $\partial_{\lambda}^2E(\varphi_{\omega}^\lambda)|_{\lambda=1}<0$, where $\varphi_\omega$ is a ground state solution of \eqref{SE}, then there exists $\varepsilon_0>0$ and $\delta_0>0$ such that, for any $v\in U_{\varepsilon_0}(\varphi_\omega)$ satisfying $\|v\|_{2}^2=\|\varphi_\omega\|_{2}^2$, we have
$$E(\varphi_\omega)\leq E(v)+(\lambda(v)-1)P(v),$$
for some $\lambda(v)\in(1-\delta_0,1+\delta_0)$.
\end{lem}
\begin{proof}
Since $\partial_{\lambda}^2E(\varphi_{\omega}^\lambda)|_{\lambda=1}<0$ and  $\partial_{\lambda}^2E(v^\lambda)$ is continuous  in $\lambda$ and $v$, we know that there exists $\varepsilon_0>0$ and $\delta_0>0$ such that $\partial_{\lambda}^2E(v^\lambda)<0$ for any $\lambda\in(1-\delta_0,1+\delta_0)$ and $v\in U_{\varepsilon_0}(\varphi_\omega)$. Notice that $\partial_\lambda E(v^\lambda)|_{\lambda=1}=P(v)$, applying the Taylor expansion for the function $E(v^\lambda)$ at $\lambda=1$,we have
\begin{equation}\label{D3}
E(v^\lambda)\leq E(v)+(\lambda-1)P(v),\ \ \lambda\in(1-\delta_0,1+\delta_0),\ \ v\in U_{\varepsilon_0}(\varphi_\omega).
\end{equation}
By  Lemma \ref{31}, we choose $\varepsilon_0<\varepsilon$ and $\delta_0<\delta$, then there exists $\lambda(v)\in(1-\delta_0,1+\delta_0)$ such that
$$I(v^{\lambda(v)})=0, \ \ \forall  v\in U_{\varepsilon_0}(\varphi_\omega).$$
Therefore,we have $$S_\omega(v^{\lambda(v)})\geq S_\omega(\varphi_\omega).$$
Since $\|v^{\lambda(v)}\|_{2}^2=\|v\|_{2}^2=\|\varphi_\omega\|_{2}^2$, we obtain
\begin{equation}\label{D4}
E(v^{\lambda(v)})=S_\omega(v^{\lambda(v)})-\frac{\omega}{2}\|v^{\lambda(v)}\|_{2}^2\geq S_\omega(\varphi_\omega)-\frac{\omega}{2}\|\varphi_\omega\|_{2}^2=E(\varphi_\omega).
\end{equation}
Thus, from\eqref{D3} and \eqref{D4},  we obtain
$$E(\varphi_\omega)\leq E(v)+(\lambda(v)-1)P(v),\ \ \forall v\in U_{\varepsilon_0}(\varphi_\omega).$$
\end{proof}
Let $\varphi_\omega$ be a ground state solution of \eqref{SE} in Lemma \ref{32}, we introduce
$$\mathcal{K}_\omega\triangleq\{v\in U_{\varepsilon_0}(\varphi_\omega);E(v)<E(\varphi_\omega),\|v\|_{2}^2=\|\varphi_\omega\|_{2}^2,P(v)<0\},$$
and
$$T(u_0)=\sup\{T>0;u(t)\in U_{\varepsilon_0}(\varphi_\omega),0\leq t\leq T\},$$
where $u(t)$ is a solution of \eqref{ME} with $u(0)=u_0$. Then, we have the following Lemma.
\begin{lem}\label{33}
If $\partial_{\lambda}^2E(\varphi_{\omega}^\lambda)|_{\lambda=1}<0$, then for any $u_0\in\mathcal{K}_\omega$, there exists $\delta_2=\delta_2(u_0)>0$ such that $P(u(t))\leq-\delta_2$ for $0\leq t\leq T(u_0)$.
\end{lem}
\begin{proof}
Take $u_0\in\mathcal{K}_\omega$ and put $\delta_1=E(\varphi_\omega)-E(u_0)>0$. From Lemma \ref{32} and $E(u(t))=E(u_0)$, we have
$$E(\varphi_\omega)\leq E(u(t))+(\lambda(u(t))-1)P(u(t))=E(u_0)+(\lambda(u(t))-1)P(u(t)),$$
which implies
\begin{equation}\label{D5}
0<\delta_1\leq(\lambda(u(t))-1)P(u(t)),\ \ 0\leq t<T(u_0).
\end{equation}
Thus,$P(u(t))\neq 0$. Since  $u_0\in\mathcal{K}_\omega$  then $P(u_0)<0$. By the continuous of $P(u(t))$ in $t$, we know
\begin{equation}\label{D6}
P(u(t))<0\ \ for\ \ 0\leq t<T(u_0).
\end{equation}
Then $\lambda(u(t))\in(1-\delta_0, 1)$, from \eqref{D5} and \eqref{D6}, we have
$$P(u(t))\leq\frac{\delta_1}{\lambda(u(t))-1}\leq-\frac{\delta_1}{\delta_0},\ \ 0\leq t<T(u_0).$$
Hence, let $\delta_2=\delta_1/\delta_0$, we have $P(u(t))\leq-\delta_2$ for $0\leq t<T(u_0)$.
\end{proof}

\hspace{-6mm}\textbf{Proof of Theorem \ref{Instability}.}
Since $P(\varphi_\omega)=\partial_\lambda S_\omega(\varphi_{\omega}^\lambda)|_{\lambda=1}=\partial_\lambda E(\varphi_{\omega}^\lambda)|_{\lambda=1}=0$ and $\partial_{\lambda}^2E(\varphi_{\omega}^\lambda)|_{\lambda=1}<0$, there exists $\delta>0$ such that
$$E(\varphi_{\omega}^\lambda)<E(\varphi_\omega)\ for\ \ \lambda\in(1,1+\delta_0).$$
On the other hand, since
$$P(\varphi_{\omega}^\lambda)=\lambda\partial_\lambda E(\varphi_{\omega}^\lambda),$$
we have
$$\partial_\lambda P(\varphi_{\omega}^\lambda)|_{\lambda=1}=P(\varphi_\omega)+\partial_{\lambda}^2E(\varphi_{\omega}^\lambda)|_{\lambda=1}=
\partial_{\lambda}^2E(\varphi_{\omega}^\lambda)|_{\lambda=1}<0.$$
Moreover, we have
$$\|\varphi_{\omega}^\lambda\|_{2}^2=\|\varphi_\omega\|_{2}^2\ \ and\ \ \lim_{\lambda\rightarrow 1}\|\varphi_{\omega}^\lambda-\varphi_\omega\|_X=0.$$
Therefore, we have
$\varphi_{\omega}^\lambda\in\mathcal{K}_\omega$ for $\lambda>1$ sufficiently close to $1$.

Since we have $|x|\varphi_{\omega}^\lambda(x)\in L^2(\mathbb{R}^3)$, from Assumption \ref{Ass}, we have
$$\frac{d^2}{dt^2}\|xu_\lambda(t)\|_{2}^2=8P(u_\lambda(t)),\ \ 0\leq t<T(\varphi_{\omega}^\lambda),$$
where $u_\lambda(t)$ is the solution of \eqref{ME} with $u_\lambda(0)=\varphi_{\omega}^\lambda.$

By Lemma \ref{33}, there exists $\delta_\lambda>0$ such that
$$P(u_\lambda(t))\leq-\delta_\lambda,\ \ 0\leq t<T(\varphi_{\omega}^\lambda).$$
Set $g(t)=\|xu_\lambda(t)\|_{2}^2>0$, the Taylor expansion at $t=0$ gives
\begin{equation}\label{D8}
g(t)\leq g(0)+g^{'}(0)t+\frac{t^2}{2}g^{''}(0)\leq g(0)+g^{'}(0)t-\frac{\delta_\lambda}{2}t^2\ \ for\ \ 0\leq t<T(\varphi_{\omega}^\lambda).
\end{equation}
This implies $T(\varphi_{\omega}^\lambda)<\infty$. Otherwise, if $T(\varphi_{\omega}^\lambda)=\infty$, by\eqref{D8}, there exists $T_1$ such that $g(T_1)<0$, this contradicts $g(t)\geq 0$ for any $t\in[0,T(\varphi_{\omega}^\lambda))$.

Combining this with the fact that $|x|\varphi_{\omega}^\lambda=|x|u_\lambda(0) \in L^2(\mathbb{R}^3)$, there exists a $T_0>0$ such that
$$\lim_{t\rightarrow T_{0}^{-}}\|x u_\lambda(t)\|_{2}^2=0.$$
Moreover, we have
$$\int_{\mathbb{R}^3}|u_\lambda|^2\leq C(\int_{\mathbb{R}^3}|x|^2|u_\lambda|^2)^\frac{1}{2}(\int_{\mathbb{R}^3}|\nabla u_\lambda|^2)^\frac{1}{2},$$
and $C$ is independent of $u_\lambda$,by the conservation of mass $\|u_\lambda\|_{2}^2=\|\varphi_{\omega}^\lambda\|_{2}^2=const>0$, we have
$$\lim_{t\rightarrow T_{0}^{-}}\|\nabla u_\lambda\|_{2}^2=+\infty,$$
and$$\|u_\lambda\|_{H^1}^2=\|u_\lambda\|_{2}^2+\|\nabla u_\lambda\|_{2}^2=+\infty.$$

\section{Stability of standing waves}
\ \ \ \ For $2-\mu/3<p<2+(2-\mu)/3$, the uniqueness of the ground state of \eqref{V0} is not know. We consider the stability of standing wave for \eqref{ME} with $\mu=1$ when $p$ is sufficiently closed to $2$. When $\mu=1$ in \eqref{Con1}, the uniqueness of $\psi_1(x)$ was investigated in \cite{U2}. Denoted by
\begin{equation}\label{F}
\aligned
&I_0(\varphi)=\|\nabla\varphi\|_{2}^2+\|\varphi\|_{2}^2-\mathfrak{F}_1(\varphi),\ \ \ \ F_0(\varphi)=\frac{1}{2p}\mathfrak{F}_1(\varphi),\\
&S_0(\varphi)=\frac{1}{2}\|\nabla\varphi\|_{2}^2+\frac{1}{2}\|v\|_{2}^2-\frac{1}{2p}\mathfrak{F}_1(\varphi).
\endaligned
\end{equation}

\subsection{Sufficient conditions for the orbital stability}

\ \ \ \ In the following, we show $\psi_1$ is the positive minimizer of
$$
\aligned
S_1&=\inf\{S_0(\varphi);\varphi\in H^1\backslash\{0\},I_0(\varphi)=0\}\\
   &=\inf\{F_0(\varphi);\varphi\in H^1\backslash\{0\},I_0(\varphi)=0\}.
\endaligned
$$
Next, we need a important result as follows.
\begin{lem}\label{CP}
Assume that $p$ satisfies the assumption of Lemma \ref{RGS} and let $\psi_1$ is the unique positive and radially symmetric solution of
$$
-\Delta\psi+\psi-(|x|^{-1}\ast|\psi|^p)|\psi|^{p-2}\psi=0,
$$
then $\psi_1$ is the minimizer of the following variational problem
\begin{equation}\label{JH}
S_1=\inf\{F_0(\varphi);\varphi\in H^1\backslash\{0\},I_0(\varphi)=0\}.
\end{equation}
Then, there exists a subsequence $\{\varphi_k\}$,
if $I_0(\varphi_k)\rightarrow 0$ and $F_0(\varphi_k)\rightarrow S_1$ as $k\rightarrow\infty$, then there exists a sequence $\{y_k\}\subset\mathbb{R}^3$ satisfying
$$
\lim_{k\rightarrow\infty}\||\varphi_k|(\cdot+y_k)-\psi_1\|_{H^1}=0.
$$
\end{lem}
\begin{proof}
Let $\{\varphi_k\}\subset H^1\backslash\{0\}$ be a bounded sequence such that $I_0(\varphi_k)\rightarrow 0$ and $F_0(\varphi_k)\rightarrow S_1$ as $k\rightarrow\infty$. We have the decomposition property for a subsequence $\{\varphi_k\}\subset H^1\backslash\{0\}$ : there exists a sequence $\{\varphi^j\}$ in $H^1$, for any $l\geq1$, we have the following identity
\begin{equation}
\varphi_k(x)=\sum_{j=1}^{l}\varphi^j(x-x_{k}^{j})+\varphi_{k}^{l}(x),
\end{equation}
with $\lim_{k\rightarrow\infty}\|\varphi_{k}^{l}\|_{p}\rightarrow0$ as $l\rightarrow\infty$ and for every $j_1\neq j_2$, $|x_{k}^{j_1}-x_{k}^{j_2}|\rightarrow\infty$ as $k\rightarrow\infty$. Moreover, as $k\rightarrow\infty$, we have
\begin{equation}\label{FJ}
\|\varphi_k\|_{H^1}^2=\sum_{j=1}^{l}\|\varphi^j\|_{H^1}^2+\|\varphi_{k}^l\|_{H^1}^2+o(1).
\end{equation}
By \eqref{F}, we know
\begin{equation}\label{PM}
\|\varphi_k\|_{H^1}^2=I_0(\varphi_k)+2pF_0(\varphi_k).
\end{equation}
From \eqref{FJ},\eqref{PM} and $k\rightarrow\infty$, we have
\begin{equation}\label{DHS}
I_0(\varphi_k)=\sum_{j=1}^{l}I_0(\varphi^j)+2p\sum_{j=1}^{l}F_0(\varphi^j)+\|\varphi_{k}^{l}\|_{H^1}^2-2pF_0(\varphi_k)+o(1).
\end{equation}
Thus we have
\begin{equation}\label{IEQ}
\sum_{j=1}^{l}I_0(\varphi^j)+2p\sum_{j=1}^{l}F_0(\varphi^j)-2pS_1\leq0,
\end{equation}
and
\begin{equation}\label{JXS}
\lim_{k\rightarrow\infty}F_0(\varphi_k)=\sum_{j=1}^{\infty}F_0(\varphi^j).
\end{equation}
According to \eqref{IEQ} and \eqref{JXS}, we derive
\begin{equation}\label{XYL}
\sum_{j=1}^{\infty}I_0(\varphi^j)\leq0.
\end{equation}

We claim that there exists exactly one $j$ such that $\varphi^j$ is nonzero. Suppose this is true, we may assume that $j=1$, then, we have
\begin{equation}\label{ZYF}
\varphi_k(x)=\varphi^1(x-x_{k}^1)+\varphi_{k}^1,
\end{equation}
with $\lim_{k\rightarrow\infty}\|\varphi_{k}^1\|_{p}\rightarrow0$ and $I_0(\varphi^1)\leq0$.
We may derive $I_0(\varphi^1)=0$. In fact, if  $\varphi^{1}\neq0$ and $I_0(\varphi^{1})<0$, then there exists some $\lambda_0\in(0,1)$ such that $I_0(\lambda_0\varphi^{1})=0$. Then, by \eqref{JH}, we have $2pS_1\leq2pF_0(\lambda_0\varphi^{1})=\|\lambda_0\varphi^{1}\|_{H^1}^2<\|\varphi^{1}\|_{H^1}^2$. However, from \eqref{FJ}, we know $\|\varphi^{1}\|_{H^1}^2\leq2pS_1$, this is a contradiction.
From \eqref{FJ}, we know $\lim_{k\rightarrow\infty}\|\varphi_{k}^{1}\|_{H^1}=0$. Therefore,
$$
\varphi_{k}(\cdot+x_{k}^1)\rightarrow\varphi^1,\ \ \hbox{in}\ \ H^1(\mathbb{R}^3),
$$
with $\varphi^1$ being a minimizer of \eqref{JH}. On the other hand, we know $|\varphi^1|$ is also a minimizer of \eqref{JH} by Kato's inequality $|\nabla|\varphi^1||\leq|\nabla\varphi^1|$. From \cite{MS1}, we know $|\varphi^1|$ is radially symmetric up to shifting the origin and a $ H^1(\mathbb{R}^3)$ solution to
\begin{equation}\label{Tem}
-\Delta\psi+\psi-(|x|^{-1}\ast|\psi|^p)|\psi|^{p-2}\psi=0
\end{equation}
with the least energy. Since $\psi_1$  is the unique positive, radial
solution of \eqref{Tem}, we know  there exists $y$ such that
$$
|\varphi^1|(x-y)=\psi_1.
$$
Let $y_k= x_{k}^1-y$, then we know
$$
\lim_{k\rightarrow\infty}\||\varphi_k|(\cdot+y_k)-\psi_1\|_{H^1}=0.
$$

Now we are ready to prove the claim  that there exists exactly one $j$ such that $\varphi^j$ is nonzero.  First repeat the same arguments as above, we can show that for every $j\geq0$,  $I_0(\varphi^j)=0$. Next if there exists two $\varphi^j\neq0$ (denoted by $\varphi^{j_1}$ and $\varphi^{j_2}$). By \eqref{FJ}, we know $\|\varphi^{j_i}\|_{H^1}^2<2pS_1$. But, \eqref{JH} implies that $2pS_1\leq2pF_0(\varphi^{j_i})=\|\varphi^{j_i}\|_{H^1}^2$, this is still a contradiction. Thus the claim is proved.

\end{proof}

\begin{lem}\label{CRS}
 Suppose that $(V0)-(V2)$ are satisfied and  $p$ satisfies the assumption of Lemma \ref{RGS}. Let $\varphi_\omega(x)\in \mathcal{M}_\omega$ and the unique positive radial solution $\psi_1$ of \eqref{V0} with $\omega=1$. Then, for the re-scaled function $\widetilde{\varphi}_\omega(x)$ defined by
$$
\varphi_\omega(x)=\omega^{\frac{1}{p-1}}\widetilde{\varphi}_\omega(\sqrt{\omega}x),
$$ we have
$$
\lim_{\omega\rightarrow\infty}\|\widetilde{\varphi}_\omega-\psi_1\|_{H^1}=0.
$$
\end{lem}
\begin{proof}
We introduce two functional as
$$
\aligned
&I_{\omega}^{*}(\varphi)=\|\nabla\varphi\|_{2}^{2}+\|\varphi\|_{2}^{2}+\omega^{-1}\int_{\mathbb{R}^3}V(\frac{x}{\sqrt{\omega}})|\varphi|^2-\mathfrak{F}_1(\varphi),\\
&S_{\omega}^{*}(\varphi)=\frac{1}{2}\|\nabla\varphi\|_{2}^{2}+\frac{1}{2}\|\varphi\|_{2}^{2}+\frac{1}{2}\omega^{-1}\int_{\mathbb{R}^3}V(\frac{x}{\sqrt{\omega}})|\varphi|^{2}-\frac{1}{2p}\mathfrak{F}_1(\varphi).
\endaligned
$$
For fixed $\omega>0$, we claim that $\widetilde{\varphi}_\omega$ is a minimizer of the variational problem as follows,
\begin{equation}\label{S1}
\aligned
S^{*}
=\inf\Big\{\frac{p-1}{2p}\mathfrak{F}_1(\varphi);\varphi\in H^1\backslash\{0\},I_{\omega}^{*}(\varphi)\leq0\Big\}.
\endaligned
\end{equation}
In fact, since
$$
S_{\omega}(v)=\frac12I_{\omega}(v)+\frac{p-1}{2p}\mathfrak{F}_1(v), \ \ v\in X,
$$
we know that
$$
\aligned
\frac{p-1}{2p}\mathfrak{F}_1(\varphi_\omega)=S_{\omega}(\varphi_\omega)&=\inf\Big\{S_{\omega}(v);v\in X\setminus\{0\},I_{\omega}(v)=0\Big \}\\
&=\inf\Big\{\frac{p-1}{2p}\mathfrak{F}_1(v);v\in X\setminus\{0\},I_{\omega}(v)=0\Big \},
\endaligned
$$
i.e.
$$
\mathfrak{F}_1(\varphi_\omega)=\inf\Big\{\mathfrak{F}_1(v);v\in X\setminus\{0\},I_{\omega}(v)=0\Big \}.
$$
Let $\Gamma_\omega:=\inf\Big\{\mathfrak{F}_1(v); v\in X\setminus\{0\},I_{\omega}(v)\leq 0 \Big\}$, it is obvious that
$$
\Gamma_\omega\leq\mathfrak{F}_1(\varphi_\omega).
$$
For any $v\in X\setminus\{0\}$ such that $I_{\omega}(v)<0$, there exits $\lambda_0\in (0,1)$ satisfying  $I_{\omega}(\lambda_0v)=0$. Consequently,
we know that
$$
\mathfrak{F}_1(\varphi_\omega)\leq \mathfrak{F}_1(\lambda_0v)< \mathfrak{F}_\mu(v),
$$
therefore
$$
\mathfrak{F}_1(\varphi_\omega)=\Gamma_\omega.
$$
Consequently, by changing variable, we know
$\widetilde{\varphi}_\omega$ minimizes \eqref{S1}.
Similarly, $\psi_1$ is the minimizer of
\begin{equation}\label{S2}
\inf\Big\{\frac{p-1}{2p}\mathfrak{F}_\mu(\varphi);\varphi\in H^1\backslash\{0\},I_0(\varphi)\leq0\Big\}.
\end{equation}
From Lemma \ref{Pro} $(1)$, we have
\begin{equation}\label{equality1}
\lim_{\omega\rightarrow\infty}\mathfrak{F}_1(\widetilde{\varphi}_\omega)
=\mathfrak{F}_1(\psi_1).
\end{equation}

Noting that  $I_{\omega}^{*}(\widetilde{\varphi}_\omega)=0$, we know
\begin{equation}\label{inequality1}
I_0(\widetilde{\varphi}_\omega)=\|\nabla\widetilde{\varphi}_\omega\|_{2}^2+\|\widetilde{\varphi}_\omega\|_{2}^2-\mathfrak{F}_1(\widetilde{\varphi}_\omega)=-\omega^{-1}\int_{\mathbb{R}^3}V(\frac{x}{\sqrt{\omega}})|\widetilde{\varphi}_\omega|^2<0.
\end{equation}
Hence, there exists a $\lambda_{0}(\omega)\in(0,1]$ such that $I_0(\lambda_{0}(\omega)\widetilde{\varphi}_\omega)=0$, then, we have
$$
\mathfrak{F}_1(\psi_1)\leq\lambda_{0}(\omega)^{2p}\mathfrak{F}_1(\widetilde{\varphi}_\omega)\rightarrow\lambda_{0}^{2p}\mathfrak{F}_1(\psi_1),\ \ \hbox{as}\ \ \omega\rightarrow\infty,
$$
for some $\lambda_{0}\in(0,1]$, which implies $\lim_{\omega\rightarrow\infty}\lambda_{0}(\omega)=1$. Thus, $\lim_{\omega\rightarrow\infty}I_{0}(\widetilde{\varphi}_\omega)=0$, we may get from \eqref{inequality1} that
$$
\lim_{\omega\rightarrow\infty}\omega^{-1}\int_{\mathbb{R}^3}V(\frac{x}{\sqrt{\omega}})|\widetilde{\varphi}_\omega|^2=0.
$$
By Lemma \ref{CP} for any sequence $\{\omega_k\}$ with $\omega_k\rightarrow\infty$, there exists a subsequence of $\{\widetilde{\varphi}_{\omega_k}\}$ and a sequence $\{y_k\}\subset\mathbb{R}^3$ such that
\begin{equation}\label{CCP}
\lim_{k\rightarrow\infty}\|\widetilde{\varphi}_{\omega_{k}}(\cdot+y_k)-\psi_1\|_{H^1}=0.
\end{equation}
Since $\widetilde{\varphi}_{\omega_{k}}\in X_G$ and the radial solution $\psi_1\in H^1(\mathbb{R}^3)$, then $y_k=0$ in \eqref{CCP}. Indeed, if $y_k\rightarrow\infty$, then $\widetilde{\varphi}_{\omega_k}\rightharpoonup0$ weakly. But $\psi_1$ is a positive radial function. It is impossible. Thus, we have
$$
\lim_{\omega\rightarrow\infty}\|\widetilde{\varphi}_\omega-\psi_1\|_{H^1}=0.
$$
\end{proof}

Related to the radial solution $\psi_1$, we may define two unbounded self-adjoint operators $L_1$ and $L_2$ from $L^2$ to $L^2$ by
$$
\aligned
&L_{1}=-\Delta+1-(p-1)(|x|^{-1}\ast|\psi_1|^p)|\psi_1|^{p-2}-p(|x|^{-1}\ast(|\psi_1|^{p-1}\cdot))|\psi_1|^{p-1},\\
&L_{2}=-\Delta+1-(|x|^{-1}\ast|\psi_1|^p)|\psi_1|^{p-2},
\endaligned
$$
with
$$
\aligned
&\langle L_1v,v\rangle=\|v\|_{H^1}^2-\mathfrak{R}_{\psi_1,v}(x,y),\\
&\langle L_2v,v\rangle=\|v\|_{H^1}^2-\int_{\mathbb{R}^3\times\mathbb{R}^3}\frac{|\psi_1(x)|^{p-2}|v(x)|^2|\psi_1(y)|^p}{|x-y|}dxdy,
\endaligned
$$
where
\begin{equation}\label{Cross}
\aligned
\mathfrak{R}_{w,v}(x,y)=(p-1)&\int_{\mathbb{R}^3\times\mathbb{R}^3}\frac{|w(x)|^{p-2}|v(x)|^2|w(y)|^p}{|x-y|}dxdy\\
+p&\int_{\mathbb{R}^3\times\mathbb{R}^3}\frac{|w(x)|^{p-2}w(x)|v(x)||w(y)|^{p-2}w(y)|v(y)|}{|x-y|}dxdy.
\endaligned
\end{equation}

We are ready to show some
\begin{lem}\label{V0L}
Let $\eta>0$ be the constant in Lemma \ref{RGS}, there exists $0<\eta'<\eta$ such that for all $p$, $2<p<2+\eta'$.\\
$(1)$ There exists $\delta_{01}>0$ such that
$$
\langle L_1v,v\rangle\geq\delta_{01}\|v\|_{2}^2,\ \ \ \ v\in H^1_G(\mathbb{R}^3,\mathbb{R}),
$$
where $(v,\psi_1)_{L^2}=0$.

\hspace{-0.6cm}$(2)$ There exists $\delta_{02}>0$ such that
$$
\langle L_2v,v\rangle\geq\delta_{02}\|v\|_{2}^2,\ \ \ \ v\in H^1(\mathbb{R}^3,\mathbb{R}),
$$
where $(v,\psi_1)_{L^2}=0$.
\end{lem}
\begin{proof}
$(1)$ By contradiction, suppose that
\begin{equation}\label{AS}
\langle L_1v_k,v_k\rangle\leq0,\ \ \ \ (v_k,\psi_1)_{L^2}=0,\ \ \ \ \|v_k\|_{H^1}=1.
\end{equation}
Since $v_k$ is bounded in $H^1_G(\mathbb{R}^3)$, there exists a subsequence $\{v_i\}$ such that $v_i\rightharpoonup v_0$ weakly in $H^1_G(\mathbb{R}^3)$ and $H^1_G(\mathbb{R}^3)\hookrightarrow L^{3}(\mathbb{R}^3)$ is compact. Thus
$$(\psi_1,v_0)_{L^2}=0$$
and
$$
\aligned
&\lim_{i\rightarrow\infty}\int_{\mathbb{R}^3}(|x|^{-1}\ast|\psi_1|^p)|\psi_1|^{p-2}v_{i}^2=\int_{\mathbb{R}^3}(|x|^{-1}\ast|\psi_1|^p)|\psi_1|^{p-2}v_{0}^2,\\
&\lim_{i\rightarrow\infty}\int_{\mathbb{R}^3}(|x|^{-1}\ast|\psi_1|^{p-1}v_i)|\psi_1|^{p-1}v_i=\int_{\mathbb{R}^3}(|x|^{-1}\ast|\psi_1|^{p-1}v_0)|\psi_1|^{p-1}v_0.
\endaligned
$$
 By lower semi-continuity, we have
\begin{equation}\label{OS}
\langle L_1v_0,v_0\rangle\leq\lim_{i\rightarrow\infty}\langle L_1v_i,v_i\rangle\leq0.
\end{equation}
We can prove that $\langle L_1v_0,v_0\rangle>0$. In fact, since
$$
S_1=\inf\{S_0(\varphi);\varphi\in H^1\backslash\{0\},I_0(\varphi)=0\},
$$
has a mountain pass characterization with  $\psi_1$ is the mountain pass solution. So the Morse index is at most one. Moreover,
$$
\langle L_1\psi_1,\psi_1\rangle=\langle L_2\psi_1,\psi_1\rangle-(p-1)\int_{\mathbb{R}^3}(|x|^{-1}\ast|\psi_1|^p)|\psi_1|^p<0.
$$
Thus, $L_1$ has exactly one negative eigenvalue $\lambda_1$ with corresponding eigenfunction $e_1$. From Theorem 1.3 of \cite{U2}, we know  there exists $0<\eta'<\eta$ such that for all $p$, $2<p<2+\eta'$, the operator $L_1$ is nondegenerate, that is
$$
\Sigma_3=Ker\{L_1\}=span\{\frac{\partial\psi_1}{\partial x_j}\},\ \ j=1,2,3.
$$
Now decomposing $H^1=\Sigma_1\oplus\Sigma_2\oplus\Sigma_3$ with $\Sigma_1=span\{e_1\}$, $\Sigma_2$ is the image of the spectral projection corresponding to the positive part of the spectrum of $L^1$.
For $\omega>0$, set $\psi_{1,\omega}=\omega^{\frac{1}{p-1}}\psi_1(\sqrt{\omega}x)$, then $\psi_{1,\omega}$ satisfies
$$
-\Delta\psi_{1,\omega}+\omega\psi_{1,\omega}-(|x|^{-1}\ast|\psi_{1,\omega}|^p)|\psi_{1,\omega}|^{p-2}\psi_{1,\omega}=0.
$$
Differential the above equation with respect to $\omega$ and take $\omega=1$ to give
$$
L_1\phi=-\Delta\phi+\phi-p(|x|^{-1}\ast(|\psi_1|^{p-1}\phi))|\psi_1|^{p-1}-(p-1)(|x|^{-1}\ast|\psi_1|^{p})|\psi_1|^{p-1}\phi=-\psi_1,
$$
where $\phi=\frac{\partial\psi_{1,\omega}}{\partial\omega}|_{\omega=1}=\frac{1}{p-1}\psi_1+\frac{1}{2}x\cdot\nabla\psi_1$. Since $\psi_1$ is radial, it is easy to check $\phi\in H_G^1$ and $(\frac{\partial\psi_1}{\partial x_j},\phi)_{L^2}=0$. We decompose $v_0$ and $\phi$ as
$$
v_0=\alpha e_1+\xi,\ \ \ \ \phi=\beta e_1+\eta,
$$
where $\alpha,\beta\in\mathbb{R}$ and $\xi,\eta\in \Sigma_2$. If $\alpha=0$, then $\langle L_1v_0,v_0\rangle=\langle L_1\xi,\xi\rangle>0$. Suppose $\alpha\neq 0$, then we have
$$
\langle L_1\phi,\phi\rangle=-(\psi_1,\phi)=-(\psi_1,\frac{1}{p-1}\psi_1+\frac{1}{2}x\cdot\nabla\psi_1)=-\frac{7-3p}{4(p-1)}\|\psi_1\|_{2}^2<0.
$$
Therefore, $\beta\neq0$ and $\langle L_1\eta,\eta\rangle=\beta^2\lambda_1+\langle L_1\phi,\phi\rangle<\beta^{2}\lambda_1$. Furthermore, since $\langle L_1\phi,v_0\rangle=-(\psi_1,v_0)_{L^2}=0=\alpha\beta\lambda_1+\langle L_1\eta,\xi\rangle$. Thus, $\langle L_1\eta,\xi\rangle=-\alpha\beta\lambda_1$. By Schwarz inequality, we have
$$
\langle L_1v_0,v_0\rangle=\alpha^2\lambda_1+(L_1\xi,\xi)\geq\alpha^2\lambda_1+\frac{|\langle L_1\eta,\xi\rangle|^2}{\langle L_1\eta,\eta\rangle}>0,
$$
this together with \eqref{OS} lead to  $v_0=0$.
However,
$$
\aligned
\lim_{i\rightarrow\infty}\langle L_1v_i,v_i\rangle&=\lim_{i\rightarrow\infty}\big(\|v_i\|_{H^1}^2-\mathfrak{R}_{\psi_1,v_i}(x,y)\big)\\
&=1-\mathfrak{R}_{\psi_1,v_0}(x,y)=1,
\endaligned
$$
which contradicts with \eqref{AS}.

$(2)$ Since $\psi_1(x)$ is the unique positive radial solution of \eqref{RRE} with $\mu=1$. We  have
$$
L_2\psi_1=-\Delta\psi_1+\psi_1-(|x|^{-1}\ast|\psi_1|^p)|\psi_1|^{p-2}\psi_1=0,
$$
and $\psi_1>0$ for $x\in\mathbb{R}^3$, $\psi_1$ is the first eigenfunction of $L^2$ corresponding to the eigenvalue $0$. Moreover, by Weyl's theorem, the essential spectrum of $L^2$ are in $[1,\infty)$, since $\psi_1$ tends to zero at infinity. These conclude $(2)$.
\end{proof}

For any $v\in X_G$ with $v_1(x)=Re\hspace{0.5mm}v(x)$ and $v_2(x)=Im\hspace{0.5mm}v(x)$, we introduce two unbounded self-adjoint operators from $L^2$ to $L^2$ defined on domain $D(-\Delta+V)$  by
$$
\aligned
&L_{\omega}^{1}=-\Delta+\omega+V-(p-1)(|x|^{-1}\ast|\varphi_\omega|^p)|\varphi_\omega|^{p-2}-p(|x|^{-1}\ast(|\varphi_\omega|^{p-1}\cdot))|\varphi_\omega|^{p-1},\\
&L_{\omega}^{2}=-\Delta+\omega+V-(|x|^{-1}\ast|\varphi_\omega|^p)|\varphi_\omega|^{p-2}.
\endaligned
$$
Therefore, $S_{\omega}^{''}$ can be expressed by
$$
\langle S_{\omega}^{''}(\varphi_\omega)v,v\rangle=\langle L_{\omega}^{1}v_1,v_1\rangle+\langle L_{\omega}^{2}v_2,v_2\rangle,
$$
and
\begin{equation}\label{EQs}
\aligned
&\langle L_{\omega}^{1}v_1,v_1\rangle=\|v_1\|_{X_\omega}^2-\mathfrak{R}_{\varphi_\omega,v_1}(x,y),\\
&\langle L_{\omega}^{2}v_2,v_2\rangle=\|v_2\|_{X_\omega}^2-
\int_{\mathbb{R}^3\times\mathbb{R}^3}\frac{|\varphi_\omega(x)|^{p-2}|v_2(x)|^2|\varphi_\omega(y)|^p}{|x-y|}dxdy,\\
&Re(\varphi_\omega,v)_{L^2}=(\varphi_\omega,v_1)_{L^2},\ \ \ \ \ \ \ \ \ \ \ \ Re(i\varphi_\omega,v)_{L^2}=(\varphi_\omega,v_2)_{L^2},
\endaligned
\end{equation}
where $\mathfrak{R}_{w,v}(x,y)$ is defined in \eqref{Cross}.

Let $\lambda_G=\inf\{\langle Hv,v\rangle;v\in X_G,\|v\|_{L^2}=1\}$, for $\omega>-\lambda_G$, we define the rescaled norm $\|\cdot\|_{\widetilde{X}_\omega}$ by
\begin{equation}\label{RN}
\|v\|^2_{\widetilde{X}_\omega}=\|v\|_{H^1}^2+\int_{\mathbb{R}^3}\omega^{-1}V(\frac{x}{\sqrt{\omega}})|v(x)|^2dx,\ \ v\in X_G.
\end{equation}
For $v(x)=\omega^{\frac{1}{p-1}}\widetilde{v}(\sqrt{\omega}x)$, we define another two re-scaled operators $\widetilde{L}_{\omega}^{1}$ and $\widetilde{L}_{\omega}^{2}$ from $L^2$ to $L^2$ by
$$
\aligned
&\widetilde{L}_{\omega}^{1}=-\Delta+1+\omega^{-1}V(\frac{x}{\sqrt{\omega}})-(p-1)(|x|^{-1}\ast|\widetilde{\varphi}_\omega|^p)|\widetilde{\varphi}_\omega|^{p-2}-p(|x|^{-1}\ast(|\widetilde{\varphi}_\omega|^{p-1}\cdot))|\widetilde{\varphi}_\omega|^{p-1},\\
&\widetilde{L}_{\omega}^{2}=-\Delta+1+\omega^{-1}V(\frac{x}{\sqrt{\omega}})-(|x|^{-1}\ast|\widetilde{\varphi}_\omega|^p)|\widetilde{\varphi}_\omega|^{p-2},
\endaligned
$$
with bilinear form as
$$
\aligned
&\langle\widetilde{L}_{\omega}^{1}v,v\rangle=\|v\|_{\widetilde{X}_\omega}^2-\mathfrak{R}_{\widetilde{\varphi}_\omega,v}(x,y),\\
&\langle\widetilde{L}_{\omega}^{2}v,v\rangle=\|v\|_{\widetilde{X}_\omega}^2-
\int_{\mathbb{R}^3\times\mathbb{R}^3}\frac{|\widetilde{\varphi}_\omega(x)|^{p-2}|v(x)|^2|\widetilde{\varphi}_\omega(y)|^p}{|x-y|}dxdy.
\endaligned
$$
By simple calculation, one can find that
\begin{equation}\label{RNO}
\aligned
&\|v\|_{X_\omega}^2=\omega^{\frac{5-p}{2(p-1)}}\|\widetilde{v}\|_{\widetilde{X}_\omega}^2,\ \ \ \ (\varphi_\omega,v)_{L^2}=\omega^{\frac{7-3p}{2(p-1)}}(\widetilde{\varphi}_\omega,\widetilde{v})_{L^2},\\
&\langle L_{\omega}^{k}v,v\rangle
=\omega^{\frac{5-p}{2(p-1)}}\langle\widetilde{L}_{\omega}^{k}\widetilde{v},\widetilde{v}\rangle,\ \ k=1,2.
\endaligned
\end{equation}
\begin{lem}\label{VL}
Assume the conditions $(V0)-(V2)$ hold and $\eta>0$ be the constant in Lemma \ref{RGS}, there exists $0<\eta'<\eta$, for $2<p<2+\eta'$ and $\varphi_\omega\in\mathcal{M}_\omega$.\\
$(1)$ There exists $\omega_1>0$ and $\delta_1>0$ such that
$$
\langle L_{\omega}^1v,v\rangle\ \geq\ \delta_1\|v\|_{X_\omega}^2,\ \ \ \ v\in\ X_G(\mathbb{R}^3,\mathbb{R})
$$
for any $\omega\in(\omega_1,\infty)$. Here, $(v,\varphi_\omega)_{L^2}=0$.

\hspace{-0.6cm}$(2)$ There exists $\omega_2>0$ and $\delta_2>0$ such that
$$
\langle L_{\omega}^2v,v\rangle\ \geq\ \delta_2\|v\|_{X_\omega}^2,\ \ \ \ v\in\ X_G(\mathbb{R}^3,\mathbb{R})
$$
for any $\omega\in(\omega_2,\infty)$. Here, $(v,\varphi_\omega)_{L^2}=0$.
\end{lem}
\begin{proof}
$(1)$  Arguing by contradiction. If  $(1)$  is not true,  from \eqref{RNO}, there exists $\omega_k\rightarrow\infty$ and ${v_k}$ satisfying $\|v_k\|_{\widetilde{X}_{\omega_k}}^2=1$ and $(v_k,\widetilde{\varphi}_{\omega_k})_{L^2}=0$ such that
\begin{equation}\label{FC}
\lim_{k\rightarrow\infty}\langle \widetilde{L}_{\omega_k}^1 v_k,v_k\rangle\leq 0.
\end{equation}

 Since $V_1(x)\geq 0$, from Lemma \ref{ES} we know
 $$
|\int_{\mathbb{R}^3}\omega_{k}^{-1}V_2(\frac{x}{\sqrt{\omega_k}})|v_k(x)|^2dx|\leq C(\omega_{k}^{\frac{3}{2q}-1}+\omega_{k}^{-1})\|V_2\|_{L^q+L^\infty}\|v_k\|_{H^1}^2.
$$
By \eqref{RN}, we get
\begin{equation}\label{CR}
\|v_k\|_{H^1}^2-C(\omega_{k}^{\frac{3}{2q}-1}+\omega_{k}^{-1})\|V_2\|_{L^q+L^\infty}\|v_k\|_{H^1}^2\leq\|v_k\|_{\widetilde{X}_{\omega_k}}^2=1,
\end{equation}
thus $\{v_k\}$ is bounded in $H^1(\mathbb{R}^3)$, if $\omega_k$ is large enough.
Moreover,
\begin{equation}\label{V20}
\lim_{k\rightarrow\infty}\int_{\mathbb{R}^3}\omega_{k}^{-1}V_2(\frac{x}{\sqrt{\omega_k}})|v_k(x)|^2dx=0.
\end{equation}
Let  $\{v_i\}$ be a subsequence of $\{v_k\}$ such that $v_i\rightharpoonup v_0$ in $H^1(\mathbb{R}^3)$ and $\{\widetilde{\varphi}_{\omega_i}\}$ be of $\{\widetilde{\varphi}_{\omega_k}\}$ such that  $\widetilde{\varphi}_{\omega_i}\rightarrow \psi_1$ in $H^1(\mathbb{R}^3)$, which is due to  Lemma \ref{CRS}. According to Lemma \ref{ES} and H\"older inequality, since $p$ is sufficiently close to $2$, we also have $|v_i|^2\rightharpoonup|v_0|^2$ in $L^{3/2}(\mathbb{R}^3)$ and $(|x|^{-1}\ast|\widetilde{\varphi}_{\omega_i}|^p)|\widetilde{\varphi}_{\omega_i}|^{p-2}\rightarrow (|x|^{-1}\ast|\psi_1|^p)|\psi_1|^{p-2}$ in $L^3(\mathbb{R}^3)$ and consequently,
\begin{equation}\label{DC1}
\lim_{i\rightarrow\infty}\int_{\mathbb{R}^3}(|x|^{-1}\ast|\widetilde{\varphi}_{\omega_i}|^p)|\widetilde{\varphi}_{\omega_i}|^{p-2}|v_i|^2dx
=\int_{\mathbb{R}^3}(|x|^{-1}\ast|\psi_1|^p)|\psi_1|^{p-2}|v_0|^2dx.
\end{equation}
By the analogous analysis, we also have
\begin{equation}\label{DC2}
\lim_{i\rightarrow\infty}\int_{\mathbb{R}^3}(|x|^{-1}\ast|\widetilde{\varphi}_{\omega_i}|^{p-2}\widetilde{\varphi}_{\omega_i}|v_i|)|\widetilde{\varphi}_{\omega_i}|^{p-2}\widetilde{\varphi}_{\omega_i}|v_i|dx
=\int_{\mathbb{R}^3}(|x|^{-1}\ast|\psi_1|^{p-2}\psi_1|v_0|)|\psi_1|^{p-2}\psi_1|v_0|dx.
\end{equation}
By \eqref{FC},\eqref{V20},\eqref{DC1},\eqref{DC2} and $V_1(x)\geq 0$, we have
$$
\aligned
0&\geq\liminf_{i\rightarrow\infty}\langle \widetilde{L}_{\omega_i}^1v_i,v_i\rangle\\
&=\liminf_{i\rightarrow\infty}\Big(\|v_i\|_{H^1}^2+\int_{\mathbb{R}^3}\omega_{i}^{-1}V(\frac{x}{\sqrt{\omega_i}})|v_i(x)|^2dx-\mathfrak{R}_{\widetilde{\varphi}_{\omega_i},v_i}\Big)\\
&\geq\|v_0\|_{H^1}^2-\mathfrak{R}_{\psi_1,v_0}=\langle L_1v_0,v_0\rangle.
\endaligned
$$
Since $(v_i,\widetilde{\varphi}_{\omega_i})_{L^2}=0$, we have $(v_0,\psi_1)_{L^2}=0$. While, according to Lemma \ref{V0L} $(1)$, we know $\langle L_1v_0,v_0\rangle>0$, then we can conclude that $v_0=0$.

However,  by \eqref{FC},\eqref{DC1} and \eqref{DC2}, we have
$$
\aligned
0&\geq\liminf_{i\rightarrow\infty}\langle \widetilde{L}_{\omega_i}^1v_i,v_i\rangle\\
&=\liminf_{i\rightarrow\infty}\Big(\|v_i\|_{\widetilde{X}_{\omega_i}}^2-\mathfrak{R}_{\widetilde{\varphi}_{\omega_i},v_i}\Big)\\
&=1-\mathfrak{R}_{\psi_1,v_0},
\endaligned
$$
which means $\mathfrak{R}_{\psi_1,v_0}\geq 1$, this contradicts with the conclusion we just proved  that  $v_0=0$. Hence, $(1)$ is concluded.

Repeat the same arguments, we can prove $(2)$.
\end{proof}

To show the stability of the standing wave solutions, we need a sufficient condition as follows.
\begin{Prop}\label{SCD1}
Assume the conditions $(V0)-(V2)$ hold and $\eta>0$ be the constant in Lemma \ref{RGS}, there exists $0<\eta'<\eta$,  for $2<p<2+\eta'$ and $\varphi_\omega\in\mathcal{M}_\omega$, $\omega\in(\omega_{0}^{*},\infty)$ where $\omega_{0}^{*}$  be the number obtained in Lemma \eqref{SCD1}. There exists $\delta^{'}>0$  such that
$$
\langle S_{\omega}^{''}(\varphi_\omega)v,v\rangle\geq\delta^{'}\|v\|_{X}^2
$$
for any $v\in X_G$ satisfying $Re(\varphi_\omega,v)_{L^2}=0$ and $Re(i\varphi_\omega,v)_{L^2}=0$.
\end{Prop}
\begin{proof}
From \eqref{EQs} and according to Lemma \ref{VL}. On the one hand, there exists $\omega_1>0$ and $\delta_1>0$ such that
$$
\langle L_{\omega}^1v_1,v_1\rangle\geq\delta_1\|v_1\|_{X_\omega}^2,\ \ \ \ \forall v_1\in X_G(\mathbb{R}^3,\mathbb{R})
$$
for any $\omega\in(\omega_1,\infty)$ and $(v_1,\varphi_\omega)_{L^2}=Re(\varphi_\omega,v)_{L^2}=0$.
On the other hand, there exists $\omega_2>0$ and $\delta_2>0$ such that
$$
\langle L_{\omega}^2v_2,v_2\rangle\geq\delta_2\|v_2\|_{X_\omega}^2,\ \ \ \ \forall v_2\in X_G(\mathbb{R}^3,\mathbb{R})
$$
for any $\omega\in(\omega_2,\infty)$ and $(v_2,\varphi_\omega)_{L^2}=Re(i\varphi_\omega,v)_{L^2}=0$.
Since there exists $\omega_0>0$, $\mathcal{M}_\omega$ is not empty. Let $\varphi_\omega(x)\in\mathcal{M}_\omega$, there exists $\omega_{0}^{*}=\max\{\omega_1,\omega_2\}$ and $\delta^{'}=\min\{\delta_1,\delta_2\}$
such that
$$
\langle S_{\omega}^{''}(\varphi_\omega)v,v\rangle=\langle L_{\omega}^{1}v_1,v_1\rangle+\langle L_{\omega}^{2}v_2,v_2\rangle\geq\delta_1\|v_1\|_{X_\omega}^2+\delta_2\|v_2\|_{X_\omega}^2\geq\delta^{'}\|v\|_{X_\omega}^2,
$$
for $\omega\in(\omega_{0}^{*},\infty)$ and any $v\in X_G$ satisfying $Re(\varphi_\omega,v)_{L^2}=0$ and $Re(i\varphi_\omega,v)_{L^2}=0$. Then the conclusion follows from the fact that $\|\cdot\|_X$ is equivalent to $\|\cdot\|_{X_\omega}$ on $X_G$.
\end{proof}

\subsection{Proof of Theorem \ref{MR}}
\ \ \ \ In this subsection, we will show the main result of stability.
For any $\varepsilon>0$ and $\varphi_\omega\in X_G$, we define
$$
U_\varepsilon(\varphi_\omega)\triangleq\{v\in X_G;\inf_{\theta\in\mathbb{R}}\|v-e^{i\theta}\varphi_\omega\|_{X_G}<\varepsilon\}.
$$
\begin{lem}\label{SCD}
Assume the conditions $(V0)-(V2)$ hold and $\eta>0$ be the constant in Lemma \ref{RGS}, there exists $0<\eta'<\eta$, for $2<p<2+\eta'$, $\varphi_\omega\in\mathcal{M}_\omega$,  $\omega\in(\omega_{0}^{*},\infty)$ where $\omega_{0}^{*}$  be the number obtained in Lemma \eqref{SCD1}. Then, there exists $C>0$ and $\varepsilon>0$ such that
$$
E(u)-E(\varphi_\omega)\geq C\inf_{\theta\in\mathbb{R}}\|u-e^{i\theta}\varphi_\omega\|_{X}^2,
$$
for $u\in U_\varepsilon(\varphi_\omega)$ with $Q(u)=Q(\varphi_\omega)$.
\end{lem}

\begin{proof}
By the implicit function theorem, if $\varepsilon>0$ is small enough, $u\in U_\varepsilon(\varphi_\omega)$ with $Q(u)=Q(\varphi_\omega)$, there exists $\theta(u)\in\mathbb{R}$ such that
\begin{equation}\label{XN}
\|u-e^{i\theta(u)}\|_{X}^2=\min_{\theta\in\mathbb{R}}\|u-e^{i\theta}\|_{X}^2.
\end{equation}
Let $v=e^{-i\theta(u)}u-\varphi_\omega$, we decompose
$$
v=a\varphi_\omega+bi\varphi_\omega+y,
$$
where $a,b\in\mathbb{R}$, and $y\in X_G$ satisfying $Re(y,\varphi_\omega)_{L^2}=0$ and $Re(y,i\varphi_\omega)_{L^2}=0$.  Obviously,
$$
\langle Q^{'}(\varphi_\omega),v\rangle=Re(\varphi_\omega,v)_{L^2}=
Re(\varphi_\omega,a\varphi_\omega+bi\varphi_\omega+y)_{L^2}=a\|\varphi_\omega\|_{2}^2.
$$
and the Taylor expansion gives
$$
Q(\varphi_\omega)=Q(u)=Q(e^{-i\theta(u)}u)=Q(\varphi_\omega+v)=Q(\varphi_\omega)+\langle Q^{'}(\varphi_\omega),v\rangle+O(\|v\|_{X}^2).
$$
Thus, we have $$a=O(\|v\|_{X}^2).$$
Moreover, we have
$$
S_\omega(u)=S_\omega(e^{-i\theta(u)}u)=S_\omega(\varphi_\omega+v)=S_\omega(\varphi_\omega)+\langle S^{'}(\varphi_\omega),v\rangle+\frac{1}{2}\langle S_{\omega}^{''}(\varphi_\omega)v,v\rangle+o(\|v\|_{X}^2),
$$
i.e.
$$
S_\omega(u)-S_\omega(\varphi_\omega)=\langle S^{'}(\varphi_\omega),v\rangle+\frac{1}{2}\langle S_{\omega}^{''}v,v\rangle+o(\|v\|_{X}^2),
$$
by $S_{\omega}^{'}(\varphi_\omega)=0$ and $Q(\varphi_\omega)=Q(u)$, we obtain
\begin{equation}\label{p1}
\aligned
E(u)-E(\varphi_\omega)&=S_\omega(u)-\omega Q(u)-(S_\omega(\varphi_\omega)-\omega Q(\varphi_\omega))\\
&=\frac{1}{2}\langle S_{\omega}^{''}(\varphi_\omega)v,v\rangle+o(\|v\|_{X}^2).
\endaligned
\end{equation}
Next, since $S_{\omega}^{'}(e^{i\theta}\varphi_\omega)=0$ for $\theta\in\mathbb{R}$, we have $S_{\omega}^{''}(\varphi_\omega)i\varphi_\omega=0$. Therefore,
\begin{equation}\label{p2}
\aligned
\langle S_{\omega}^{''}(\varphi_\omega)y,y\rangle&=\langle S_{\omega}^{''}(\varphi_\omega)v,v\rangle-2a\langle S_{\omega}^{''}(\varphi_\omega)\varphi_\omega,v\rangle+a^2\langle S_{\omega}^{''}(\varphi_\omega)\varphi_\omega,\varphi_\omega\rangle\\
&=\langle S_{\omega}^{''}(\varphi_\omega)v,v\rangle+O(\|v\|_{X}^3).
\endaligned
\end{equation}
Since $y\in X_G$ satisfies $Re(y,\varphi_\omega)_{L^2}=0$ and $Re(y,i\varphi_\omega)_{L^2}=0$, by Proposition \ref{SCD1}, there exists $\delta^{'}>0$ such that
\begin{equation}\label{p3}
\langle S_{\omega}^{''}(\varphi_\omega)y,y\rangle\geq\delta^{'}\|y\|_{X}^2.
\end{equation}
From \eqref{p2} and \eqref{p3}, we know
\begin{equation}\label{po}
\langle S_{\omega}^{''}(\varphi_\omega)v,v\rangle\geq \delta^{'}\|y\|_{X}^2-O(\|v\|_{X}^3).
\end{equation}
Now, by \eqref{XN} and $(\varphi_\omega,i\varphi_\omega)_{X}=0$, we have $$0=(v,i\varphi_\omega)_{X}=b\|\varphi_\omega\|_{X}^2+(y,i\varphi_\omega)_{X}.$$ Thus, we have $|b|\|\varphi_\omega\|_X\leq\|y\|_X$ and $\|v\|_X\leq(|a|+|b|)\|\varphi_\omega\|_X+\|y\|_X\leq2\|y\|_X+O(\|v\|_{X}^2)$. Therefore, we have
\begin{equation}\label{p4}
\|y\|_{X}^2\geq\frac{1}{4}\|v\|_{X}^2+O(\|v\|_{X}^3).
\end{equation}
By \eqref{p1}, \eqref{po} and \eqref{p4}, we have
$$
E(u)-E(\varphi_\omega)\geq\frac{\delta^{'}}{2}\|y\|_{X}^2+o(\|v\|_{X}^2)\geq\frac{\delta^{'}}{8}\|v\|_{X}^2+o(\|v\|_{X}^2).
$$
Thus, for $u\in U_\varepsilon(\varphi_\omega)$ and $\|v\|_X=\|u-e^{i\theta(u)}\varphi_\omega\|_X<\varepsilon$, we may take $\varepsilon=\varepsilon(\delta^{'})>0$ small enough to obtain
$$
E(u)-E(\varphi_\omega)\geq\frac{\delta^{'}}{8}\|u-e^{i\theta(u)}\varphi_\omega\|_{X}^2.
$$
\end{proof}

\textbf{\large Proof of Theorem \ref{MR}.}
Argue by contradiction. If the standing wave is unstable, then there exists a sequence of initial data $u_n(0)$ and $\delta>0$ such that
$$
\inf_{\theta\in\mathbb{R}}\|u_n(0)-e^{i\theta}\varphi_\omega\|_{X}\rightarrow0,
$$
but
$$
\sup_{t>0}\inf_{\theta\in\mathbb{R}}\|u_n(t)-e^{i\theta}\varphi_\omega\|_{X}\geq\delta,
$$
where $u_n(t)$ is a solution with initial value $u_n(0)$. By continuity in $t$, we can pick the first time $t_n$ so that
\begin{equation}\label{RESULT}
\inf_{\theta\in\mathbb{R}}\|u_n(t_n)-e^{i\theta}\varphi_\omega\|_{X}=\delta.
\end{equation}
 By Proposition \ref{CAUP}, $E$ and $Q$ are conserved in $t$. Then, we have
$$
E(u_n(t_n))=E(u_n(0))\rightarrow E(\varphi_\omega),
$$
$$
Q(u_n(t_n))=Q(u_n(0))\rightarrow Q(\varphi_\omega).
$$
There exists a sequence $\{v_n\}$ such that $\|v_n-u_n(t_n))\|_X\rightarrow0$ and $Q(v_n)=Q(\varphi_\omega)$. Because of the continuity of $E$, we have $E(v_n)\rightarrow E(\varphi_\omega)$. If we choose $\delta$ small enough, from Lemma \ref{SCD}, we can obtain
$$
c\|v_n-e^{-i\theta(v_n)}\varphi_\omega\|_{X}^2=c\|e^{i\theta(v_n)}v_n-\varphi_\omega\|_{X}^2 \leq E(v_n)-E(\varphi_\omega)\rightarrow0.
$$
Thus, $\|u_n(t_n)-e^{-i\theta(v_n)}\varphi_\omega\|_X\rightarrow0$, which contradicts \eqref{RESULT}.

\end{document}